\documentclass[11pt,notitlepage]{article}
\usepackage[a4paper,margin=1.05in]{geometry}
\usepackage{graphicx,wrapfig,graphics, amsfonts, amssymb, mathtools, bbm, bm, amsthm, float, dsfont, array}
\usepackage[toc,page]{appendix}
\usepackage[hidelinks]{hyperref}
\usepackage[english]{babel}
\usepackage{float}
\usepackage[super]{nth}
\usepackage{titling}
\usepackage{caption}
\usepackage{subcaption,dsfont}
\usepackage{multirow}
\usepackage{pbox}
\usepackage{fancyhdr, pdfpages, lipsum, accents}
\allowdisplaybreaks
\usepackage{authblk}

\newtheorem{theorem}{Theorem}[section]

\newtheorem{proposition}[theorem]{Proposition}
\newtheorem{lemma}[theorem]{Lemma}

\newtheorem{remark}[theorem]{Remark}
\newtheorem{settings}[theorem]{Setting}

\newcommand{\E}{\mathbb{E}}
\newcommand{\R}{\mathbb{R}}
\newcommand{\N}{\mathbb{N}}
\newcommand{\Pro}{\mathbb{P}}
\newcommand{\X}{\mathbf{X}}
\newcommand{\x}{\mathbf{x}}

\newcommand{\Z}{\mathbf{Z}}
\newcommand{\W}{\mathbf{W}}

\newcommand{\Var}{\mathrm{Var}}
\numberwithin{equation}{section}
\title{Optimal Liquidation in a Mean-reverting Portfolio}
\author[1]{Riccardo Cesari\thanks{Department of Mathematics, Imperial College, London SW7 2BZ, UK. \href{riccardo.cesari17@imperial.ac.uk}{riccardo.cesari17@imperial.ac.uk} } \ and   Harry Zheng\thanks{Department of Mathematics, Imperial College, London SW7 2BZ, UK. \href{h.zheng@imperial.ac.uk}{h.zheng@imperial.ac.uk} }}

\begin{document}
\maketitle
\noindent\textbf{Abstract}. In this work we study a finite horizon optimal liquidation problem with multiplicative price impact in algorithmic trading, using market orders. We analyze the case when an agent is trading on a market with two financial assets, whose difference of log-prices is modelled with a mean-reverting process. The agent's task is to liquidate an initial position of shares of one of the two financial assets, without having the possibility of trading the other stock. The criterion to be optimized consists in maximising the expected final value of the agent, with a running inventory penalty.
The main result of this paper consists in finding a classical solution of the Hamilton-Jacobi-Bellman (HJB) equation associated to this problem, which is proved to not coincide with the value function. However, we find the value function as a solution to the forward-backward stochastic differential equation (FBSDE) associated to the problem. We provide numerical tests showing that the HJB and FBSDE solutions are close to each other and analysing performance of the described model. We also prove a verification theorem and a comparison principle for the viscosity solution to the HJB equation.\\




\section{Introduction}
\label{sectionintro}
A standard service of investment banks is the execution of large trades. Unlike for small trades, the liquidation of a large portfolio is a complex task. It is usually impossible to immediately execute a large liquidation task or it is only possible at a high cost due to insufficient liquidity. Hence, the ability of exercising an order in a way that minimizes execution costs for the client is of primary importance. The objective of this paper is to determine the adaptive trading strategy that maximizes the expected final cash value of an asset sale. We address this question in the continuous-time liquidity model introduced by Almgren~\cite{Almgren2003} with an infinite time horizon and linear price impact. 

The optimal liquidation problem under price impact has been studied extensively in the literature. Bertsimas and Lo~\cite{Bertsimas1998} use a linear price impact model and solve a discrete optimal control problem to minimize expected trading costs. Almgren and Chriss~\cite{Almgren1999, Almgren2001}, Huberman and Stanzl~\cite{Huberman2005} introduce the volatility as a trading cost. Almgren~\cite{Almgren2003} employs nonlinear impact functions and discusses the continuous-time limit of the models in Almgren and Chriss~\cite{Almgren1999, Almgren2001} in more details. Almgren~\cite{Almgren2012} considers optimal liquidation in a market with stochastic liquidity and stochastic volatility. Kharroubi and Pham~\cite{Kharroubi} consider real trading that occurs in discrete time. Obizhaeva and Wang~\cite{Obizhaeva2013} include price impact by modelling the limit order book directly. For an overview of continuous-time price impact models, see Cartea et al.~\cite{Cartea} and the references therein.

All the literature we have inspected on the optimal liquidation strategy is based only on the stock that the agent needs to liquidate. However, there may be additional information available in the market, such as the price of a correlated stock, which could be helpful to better predict the stock price movements. A model based on both asset prices may generate a more reliable adaptive liquidation strategy, which not only relies on the price of the liquidating stock, but also on that of the correlated stock.

In this paper we analyze the case when an agent trades on a market with two financial assets whose difference of log-prices has a mean-reverting behavior. The agent's task is to liquidate the initial position of shares of one stock, without the possibility of trading the other stock. This technique is often employed when modeling a pair of stocks in pair trading in which the agent tries to make money out of a couple of correlated stocks by selling one stock and buying the other, to take advantage of the mean-reverting behavior of the co-integration factor between the two stocks. In our setting the agent can only sell stock, but cannot trade the other stock. Moreover, we define the difference of the log-prices to be an Ornstein-Uhlenbeck process which is the continuous-time analogue of the discrete-time AR(1) process and makes its parametrization an easy task, see Cartea et al.~\cite[Section 3.7]{Cartea} and Brockwell and Davis~\cite[Chapter 3]{Davis} for further details on parametrization of such processes.

The main contributions of this paper are that we prove the value function is the unique continuous viscosity solution to the HJB equation which is complicated with three state variables, that we find the representation of the classical solution under some mild conditions, which opens the way of finding the optimal value and strategy with the Monte-Carlo simu- lation, and that we show the value function and the optimal liquidation rate depend only on observable data which allow a straightforward calculation at each moment in time. Although the classical solution to the HJB equation is proved to be not coincident with the value function, numerical tests show that it is close to the value functionn, by proving that it is close to the approximated solution of the FBSDE associated to the optimization problem.

The rest of the paper is structured as follows. Section \ref{sectionmodel} describes the settings of the problem, defines the value function, writes the HJB equation and states the main theorem (Theorem \ref{thmviscsolgen1}) that the value function is the unique continuous viscosity solution to the HJB equation. Section \ref{sectionclass} finds a classical solution of the HJB equation under some mild conditions on the model parameters and derives the objective function as a sum of classical solutions to three different parabolic PDEs which can be solved one by one. Section \ref{sectionFBSDE} finds value function and optimal trading speed as solution to an FBSDE obtained by applying stochastic maximum principle to our problem. Section \ref{sectionnumeric} is the numerical section and it is divided in two parts. Subsection \ref{subsectionNN} compares the closed form solution obtained in section \ref{sectionclass} with the solution of the FBSDE in section \ref{sectionFBSDE}, which is approximated using a deep learning algorithm. Subsection \ref{subsectioncloseform} provides some numerical tests to assess our model and compares its performance with that of two other strategies based on two geometric Brownian motion approximations of the liquidating stock price. Section \ref{sectionconclusion} concludes. Appendix \ref{sectionproof} contains the proofs of Theorem \ref{thmviscsolgen1}, Propositions \ref{propassu} and \ref{remarkonverif}.

\section{Model}
\label{sectionmodel}
Let $(\Omega, \mathcal{F}, (\mathcal{F}_t)_{t\in [0,T]}, \Pro)$ be a filtered probability space, where $(\mathcal{F}_t)_{t\in [0,T]}$ is the natural filtration generated by two independent standard Brownian motions $W^1$ and $W^2$, augmented by all $\Pro$-null sets. Let $T$ be the fixed terminal time, $(A_r)_{r\in[0,T]}$ the price of a stock in the market, satisfying the following geometric Brownian motion (GBM):
\begin{equation}
\label{Awritten}
dA_r = \mu_1 A_r dr+\sigma_1 A_r dW^1_{r}, \quad A_0=a,
\end{equation}
where $\mu_1$, $\sigma_1$ are positive constants, $\mu_1$ is the growth rate, $\sigma_1$ the volatility rate, $(S_r)_{r\in[0,T]}$ the price of the stock that the agent aims to liquidate, $(\varepsilon_r)_{r\in[0,T]}$ the cointegration factor between stocks $S_r$ and $A_r$, defined by $\varepsilon_r = \ln\left(\frac{S_r}{A_r}\right)$, and follows an Ornstein-Uhlenbeck (OU) process
\begin{equation}
\label{epsilonwritten}
d\varepsilon_r =-k\varepsilon_r dr+\sigma_2 \left(\rho dW^1_r+\sqrt{1-\rho^2} dW^2_r\right), \quad \varepsilon_0=\epsilon,
\end{equation}
where $k$, $\sigma_2$ are positive constants, $\rho\in [-1,1]$ the correlation coefficient, $k$ the mean reversion speed, $\sigma_2$ the volatility. The cointegration factor $\varepsilon_r$ behaves as a mean-reverting process, which implies a period of time in which the process $S_r$ outperforms (or underperform) $A_r$ is followed by a moment in which the two stocks have similar prices.

Let $(c_r)_{r\in[0,T]}$ denote the rate of selling the stock, which is a decision (control) variable decided by the agent and is said admissible if it is a progressively measurable, non-negative and square integrable process. Denote by $\mathcal{A}$ the set of all admissible control processes.

Let $(Q_r)_{r\in[0,T]}$ denote the inventory left at time $r$ and $q_0>0$ the initial amount of stock owned by the agent. The process $Q_r$ depends on the trading strategy $c$ and follows the equation:
\begin{equation}
\label{qwritten}
dQ_r=-c_r  dr, \quad  Q_0=q_0.
\end{equation}
Let $(M_r)_{r\in[0,T]}$ denote the wealth process, satisfying the following equation:
\begin{equation*}
dM_r= c_r(S_r-\eta c_r)dr, \quad M_0=0,
\end{equation*}
where $\eta\ge 0$ is the temporary price impact factor, which is the same as that in Cartea et al.~\cite{Cartea}.

Denote by $\x$ the vector of three state variables $(a,\epsilon,q)$ and $\mathcal{O}$ the state space, given by $\mathcal{O}:=(0,\infty)\times \R \times [0,\bar{Q}_0)$ with $q_0<\bar{Q}_0<\infty$. Moreover, denote the initial price of the stock $S$ by $s:=ae^{\epsilon}$. We group the three state processes $(A,\varepsilon,Q)$ into a vector $\X$. Let $t\in[0,T]$, we define the 3-dimensional stochastic process $(\X_r)_{r\in[t,T]}:=\left(A_r, \varepsilon_r, Q_r \right)_{r\in[t,T]}$ as the solution to the following SDE
\begin{equation}
\label{SDEdef}
d\X_r=\mu(\X_r,c(r,\X_r))dt+\sigma(\X_r)d\mathbf{W}_r,
\end{equation}
where
\begin{align}
\label{defmusigma}
\mu(\x,c)=\begin{pmatrix}
\mu_1\\
-k \epsilon\\
-c
\end{pmatrix}, \quad \sigma(\x)= \begin{pmatrix}
\sigma_1 a&0\\
\sigma_2 \rho&\sigma_2 \sqrt{1-\rho^2}\\
0&0
\end{pmatrix}.
\end{align}

The optimal liquidation problem is defined by:
\begin{equation}
\label{eqmax}
\sup_{c\in \mathcal{A}}\mathbb{E}\bigg[ M_\tau+ Q_\tau \left(S_\tau-\chi Q_\tau \right) -\phi_1 \int_0^\tau Q_r^2 \;dr-\phi_2 \int_0^\tau S_r Q_r \;dr-\phi_3 \int_0^\tau A_r Q_r \;dr \bigg],
\end{equation}
where $\tau$ is a stopping time defined by
$\tau=T\wedge \min\{r\ge 0\; | \; Q_r=0\}$, the first time when all stock is liquidated before terminal time $T$ or $T$ otherwise. The first term inside expectation is the wealth value at $\tau$, the second the terminal liquidation value and the last three the running inventory penalties. The terminal liquidation value is the cash from liquidating all the inventory left at terminal time $T$ at a price $S_T$ penalized by a quantity proportional to the amount of remaining stocks. Inventory penalties are not financial costs, but incorporate the agent's urgency for executing the trade. Denote by $\E_t[\cdot]=\E[\cdot | A_t=a,\ \varepsilon_t=\epsilon, \ Q_t=q]$, the conditional expectation operator at time $t\in[0,T]$.

The value function of problem $\eqref{eqmax}$ is defined by
\begin{equation}
\label{1valfunc}
v(t,a,\epsilon,q)=\sup_{c\in \mathcal{A}} v^{c}(t,a,\epsilon,q),
\end{equation}
where
\begin{align}
\label{1eqq15}
v^{c}(t,a,\epsilon,q)&= \mathbb{E}_t\bigg[ M_\tau+ Q_\tau \left(S_\tau-\chi Q_\tau \right) -\phi_1 \int_t^\tau Q_r^2 \;dr-\phi_2 \int_t^\tau S_r Q_r \;dr-\phi_3 \int_t^\tau A_r Q_r \;dr \bigg],
\end{align}
where $\tau$ is defined by $\tau=T\wedge \min\{r\ge t\; | \; Q_r=0\}$.

To solve the control problem $\eqref{1valfunc}$, we adopt the dynamic programming principle and derive the following HJB equation for the value function:
\begin{equation}
\label{1eqq3tris3}
\frac{\partial w}{\partial t} +\mathcal{L} w +\sup_{c\ge 0} \left[ -c\frac{\partial w}{\partial q}+ae^{\epsilon} c-\eta c^2 \right]-\phi_1 q^2 -\phi_2q ae^{\epsilon}-\phi_3 qa =0
\end{equation}
on $[0,T)\times \mathcal{O}$, with terminal condition $w(T,a,\epsilon,q)= q\left(ae^{\epsilon}-\chi q \right)$ and boundary condition $w(t,a,\epsilon,0)=0$, where $\mathcal{L}$ is the operator defined by
\begin{equation*}
\mathcal{L} w= \frac{\sigma_1^2}2 a^2 \frac{\partial^2 w}{\partial a^2}+\rho\sigma_1\sigma_2 a \frac{\partial^2 w}{\partial a \partial \epsilon}+\frac{\sigma_2^2}2 \frac{\partial^2 w}{\partial \epsilon^2}+\mu_1 a \frac{\partial w}{\partial a} -k \epsilon \frac{\partial w}{\partial \epsilon}.
\end{equation*}

\begin{theorem}[Verification Theorem]
\label{vertheo}
Let $w$ be a function in $C^{1,2}([0,T)\times \mathcal{O})\cap C^0([0,T]\times \bar{\mathcal{O}})$ and satisfy the following growth condition
$$|w(t,\x)|\le C(1+q^2)(1+a^{p_1})(1+e^{p_2\epsilon}) \qquad \forall (t,\x)\in [0,T]\times \mathcal{O}$$
for fixed $p_1,p_2,C>0$. Assume there exists a measurable function $c^*(t,\x)$ such that 
\begin{align*}
&\frac{\partial w}{\partial t} +\mathcal{L} w +\sup_{c\ge 0} \left[ -c\frac{\partial w}{\partial q}+ae^{\epsilon} c-\eta c^2 \right]-\phi_1 q^2 -\phi_2q ae^{\epsilon}-\phi_3 qa \\
&\quad= \frac{\partial w}{\partial t} +\mathcal{L} w -c^*\frac{\partial w}{\partial q}+ae^{\epsilon} c^*-\eta (c^*)^2 -\phi_1 q^2 -\phi_2q ae^{\epsilon}-\phi_3 qa=0
\end{align*}
with terminal condition $w(T,a,\epsilon,q)= q\left(ae^{\epsilon}-\chi q \right)$ and boundary condition $w(t,a,\epsilon,0)=0$. Let the SDE
\begin{equation}
\label{SDEver}
d\X_r=\mu(\X_r,c^*(r,\X_r))dt+\sigma(\X_r)d\mathbf{W}_r
\end{equation}
admit a unique solution, given an initial condition $\X_t=\x$, where $\mu$ and $\sigma$ are defined in $\eqref{defmusigma}$. Let $(c^*(r,\X_r))_{r\in[t,T]}\in \mathcal{A}$. Then $w$ coincides with the value function $v$.
\end{theorem}

Equation $\eqref{1eqq3tris3}$ is a nonlinear PDE with three state variables $a,\epsilon$ and $q$. We show the value function is a viscosity solution of $\eqref{1eqq3tris3}$, see Pham~\cite{Pham} for its definition and properties.
\begin{theorem}
\label{thmviscsolgen1}
The value function $v$ defined in $\eqref{1valfunc}$ is the unique viscosity solution of the HJB equation $\eqref{1eqq3tris3}$.
\end{theorem}

If we strengthen the condition on the control set, we have continuity of the value function. Let $(t,\x)\in [0,T]\times \mathcal{O}$ be fixed and let $\gamma,N>0$. Then, we define the set $\tilde{\mathcal{A}}_{\gamma,N}(t,\x)$ as
\begin{equation}
\label{defAtilde}
\tilde{\mathcal{A}}_{\gamma,N}(t,\x) = \bigg\{c\in\mathcal{A}(t,\x) \bigg| \ \left(\E\left[\int_t^T c_r^{2+\gamma} \; dr\right]\right)^{\frac 1{2+\gamma}}\le N(1+a)\left(1+e^{N \epsilon}\right)\bigg\}.
\end{equation}
\begin{proposition}
\label{propAtilde}
Let the set of admissible controls be reduced to $\tilde{\mathcal{A}}_{\gamma,N}$ for fixed $\gamma,N>0$. Then the value function $v$, defined in $\eqref{1valfunc}$, is continuous on $[0,T]\times \mathcal{O}$.
\end{proposition}

\section{Classical solution to HJB equation $\eqref{1eqq3tris3}$}
\label{sectionclass}
It is in general difficult to find a classical solution of equation $\eqref{1eqq3tris3}$. We show in this section that, under some mild conditions, we can achieve that. From equation $\eqref{1eqq3tris3}$ we get the optimal rate of trading as
\begin{equation}
\label{eqq4}
c^*(t,a,\epsilon,q)=\frac 1{2\eta} \max \left\{ae^{\epsilon}-\frac{\partial w}{\partial q}, \ 0\right\}.
\end{equation}
Substituting $c^*$ in equation $\eqref{1eqq3tris3}$, we have
\begin{equation}
\label{eqq3bis}
\frac{\partial w}{\partial t} + \mathcal{L} w -\phi_1 q^2-\phi_2 q ae^{\epsilon}-\phi_3 qa+\frac{1}{4\eta} \left(\max\Big\{ae^{\epsilon}-\frac{\partial w}{\partial q} , \ 0\Big\}\right)^2=0.
\end{equation}
The PDE $\eqref{eqq3bis}$ is nonlinear and difficult to solve. To simplify it we try to eliminate the last term containing the $\max$ operator. In the following, we assume that the function $w$ satisfies $ae^{\epsilon}-\frac{\partial w}{\partial q} \ge 0$ for any $(t,\x)\in [0,T]\times\mathcal{O}$.
Under this assumption, $\eqref{eqq3bis}$ reduces to
\begin{equation}
\label{eqq2}
\frac{\partial w}{\partial t} +\mathcal{L} w -\phi_1 q^2-\phi_2 qae^{\epsilon}-\phi_3 qa+\frac{1}{4\eta} \left(ae^{\epsilon}-\frac{\partial w}{\partial q}\right)^2=0.
\end{equation}
By inspecting the terminal conditions, we postulate a solution of the following form:
\begin{equation}
\label{eqq6}
w(t,a,\epsilon,q)=g_1(t,\epsilon,a)\mathds{1}_{\{q>0\}} +qae^{\epsilon} g_2(t,\epsilon)+q^2 g_3(t).
\end{equation}
Substituting $\eqref{eqq6}$ to equation $\eqref{eqq2}$, collecting terms with coefficients $1, \ sq$ and $q^2$, and setting each term equal to $0$, we derive the following system of PDEs on $[0,T]\times (\mathcal{O}\cap \{q>0\})$:
\begin{equation}
\label{HJBsplitted}
\begin{cases}
0&=\frac{\partial g_1}{\partial t} + \mathcal{L} g_1+\frac{1}{4\eta} a^2e^{2\epsilon}\left(1-g_2\right)^2,\\

0&=\frac{\partial g_2}{\partial t} +\frac{\sigma_2^2}2 \frac{\partial^2 g_2}{\partial \epsilon^2}+\left(\sigma^2_2+\rho\sigma_1\sigma_2-k \epsilon\right) \frac{\partial g_2}{\partial \epsilon} +\left( \frac{\sigma_2^2}2+\rho\sigma_1\sigma_2+\mu_1 -k \epsilon+\frac {g_3}{\eta}\right) g_2 -\phi_2-\phi_3 e^{-\epsilon}- \frac {g_3}{\eta},\\

0&=g_3 ' -\phi_1 +\frac{g_3^2}{\eta},
\end{cases}
\end{equation}
with terminal conditions $g_1(T,a,\epsilon)=0, \ g_2(T,\epsilon)=1, \ g_3(T)=- \chi$.

The last equation in $\eqref{HJBsplitted}$ is a Riccati type equation and has a closed form solution given by
\begin{equation}
\label{g3}
g_3(t)=\sqrt{\phi_1 \eta} \frac{e^{2t\sqrt{\frac{\phi_1}{\eta}}}(\sqrt{\phi_1 \eta}- \chi)-e^{2T\sqrt{\frac{\phi_1} {\eta}}}(\sqrt{\phi_1 \eta}+ \chi)}{e^{2t\sqrt{\frac{\phi_1}{\eta}}}(\sqrt{\phi_1 \eta}- \chi)+e^{2T\sqrt{\frac{\phi_1}{\eta}}}(\sqrt{\phi_1 \eta}+\chi)} , \qquad \forall t\in[0,T].
\end{equation}
It is easy to  verify that $g_3$ is a negative and increasing function.

Recall that function $w$ must satisfy $ae^{\epsilon}-\frac{\partial w}{\partial q} \ge 0$ for any $(t,\x)\in [0,T]\times\mathcal{O}$, which is equivalent to the following:
\begin{equation}
\label{eqqfirstcond}
ae^{\epsilon}\left( 1-g_2(t,\epsilon )\right)-2qg_3(t)\ge 0, \quad \forall (t,a,\epsilon,q)\in[0,T]\times \mathcal{O}.
\end{equation}
Since $a$ is positive and $g_3$ is negative, condition $\eqref{eqqfirstcond}$ holds if
\begin{equation}
\label{condong2}
g_2(t,\epsilon)\le 1, \quad \forall (t,\epsilon)\in[0,T]\times \R.
\end{equation}

\begin{proposition}
\label{propassu}
Assume the model parameters satisfy the following condition:
\begin{equation}
\label{eqq51}
\phi_3 e^{1+\frac{\phi_2}{k}}\ge ke^{\frac{ \sigma_2^2}{2k}+\frac{\mu_1}{k}+\frac{\rho}{k}\sigma_1\sigma_2}.
\end{equation}
Then solution $g_2$ in $\eqref{HJBsplitted}$ satisfies condition $\eqref{condong2}$ and is given by
\begin{align}
\nonumber
g_2(t,\epsilon)=1-\phi_3 e^{-\epsilon}\int_t^T\hat{g}(r;t)\;dr&-e^{-\epsilon} \int_t^T \hat{g}(r;t)e^{\bar{\mu}(r-t) \epsilon+\frac{\bar{\sigma}(r-t)^2}{2}+\frac{\rho}{k}\sigma_1\sigma_2(1-\bar{\mu}(r-t))} \cdot\\
\label{eqq25}
&\qquad\cdot\left( k\bar{\mu}(r-t) \epsilon+k\bar{\sigma}(r-t)^2-\frac{ \sigma_2^2}{2} -\rho\sigma_1\sigma_2\bar{\mu}(r-t)-\mu_1+\phi_2\right)\;dr
\end{align}
for $(t,\epsilon)\in [0,T]\times \R$, where $\bar{\mu}$, $\bar{\sigma}$ and $\hat{g}$ are functions defined by
\begin{align}
\label{eqqmu}
\bar{\mu}(s)=e^{-ks},\quad \bar{\sigma}(s)^2=\frac{\sigma_2^2}{2k}\left(1-e^{-2ks}\right), \quad \hat{g}(r;t)=\exp\left(\frac{1}{\eta} \int_t^r g_3(s)\; ds +\mu_1(r-t)\right).
\end{align}
Moreover, the optimal control is given by
\begin{equation}
\label{eqq7bis}
c^*(t,a,\epsilon,q)= \frac 1{2\eta}\left[ae^{\epsilon}\left( 1-g_2(t,\epsilon )\right)-2qg_3(t) \right] \mathds{1}_{q>0} .
\end{equation}
\end{proposition}

Note that for any fixed parameters $k,\sigma_1,\sigma_2,\mu_1, \rho$, one can always choose $\phi_2$ and $\phi_3$ such that $\eqref{eqq51}$ is satisfied, and that $g_1$ in $\eqref{HJBsplitted}$ can be written, with the help of the Feynman-Kac formula, as
\begin{equation}
\label{eqqg1}
g_1(t,a,\epsilon)=\frac{1}{4\eta}  \mathbb{E}_t\Bigg[ \int_t^T S_r^2\left(1-g_2(r, \varepsilon_r)\right)^2\; dr \Bigg].
\end{equation}
Combining $\eqref{eqqg1}$ with Proposition \ref{propassu}, we have the following result:

\begin{theorem}
\label{thmbig2}
Assume condition $\eqref{eqq51}$ is satisfied. Then equations in $\eqref{HJBsplitted}$ admit classical solutions $g_1$, $g_2$ and $g_3$ given by $\eqref{eqqg1}$, $\eqref{eqq25}$ and $\eqref{g3}$ respectively.
\end{theorem}

\begin{proposition}
\label{propcstar}
Let $c^*$ be defined as in $\eqref{eqq7bis}$, let $\gamma$ be any positive real number, let $N>0$ be big enough and let $(t,\x)\in[0,T]\times \mathcal{O}$. Then, $c^*(r,\X_r) \in \tilde{\mathcal{A}}_{\gamma,N}(t,\x)$, as defined in $\eqref{defAtilde}$.
\end{proposition}

\begin{remark}
\label{exanduniqsde}
Let $c^*$ be defined as in $\eqref{eqq7bis}$. By the fact that $g_2$ is locally Lipschitz, we get that $c^*$, drift and diffusion coefficients of SDE $\eqref{SDEver}$ are locally Lipschitz. Existence and uniqueness of solution to equation $\eqref{SDEver}$ follow by applying Karatzas and Shreve~\cite[Theorem 2.5]{Kar}.
\end{remark}

\begin{remark}
Let condition $\eqref{eqq51}$ be satisfied and $g_1$, $g_2$ and $g_3$ be the classical solutions to the equations in $\eqref{HJBsplitted}$ as in Theorem \ref{thmbig2}. As proved in Remark \ref{exanduniqsde} and Proposition \ref{propcstar}, all conditions of verification Thorem \ref{vertheo} are satisfied except for continuity of $w$. Indeed function $w$ is not continuous for $q\to 0$, unless $g_1\equiv 0$ on $[0,T]\times \mathcal{O}$. Hence, $w$ does not necessarily coincide with the value function $v$ in $\eqref{1valfunc}$, unless it is proved to be continuous on $q=0$.
\end{remark}

As proved in Proposition \ref{propcstar}, the optimal control $c^*$ lies in the more restrictive control set $\tilde{\mathcal{A}}_{\gamma,N}$ defined in $\eqref{defAtilde}$. Proposition \ref{propAtilde} ensures that if the control set is reduced to $\tilde{\mathcal{A}}_{\gamma,N}$, then the value function $v$ is continuous. We conclude that if $g_1$ is not identically equal to $0$, then the solution $w$ to the HJB equation does not coincide with the value function.

If we reduce our model to a one stock model as that in Cartea et al.~\cite{Cartea} with the same  parameters, i.e.,   $\epsilon_0=0, \sigma_2=0, \ \mu_1=0, \ \rho=0, \ k=0, \ \phi_2=\phi_3=0$,  then condition $\eqref{eqq51}$ is satisfied. Using $\eqref{eqq25}$ and $\eqref{eqqg1}$, we get $g_2(t,\epsilon)=1$ and $g_1(t,a,\epsilon)=0$ for any $(t,a,\epsilon)\in[0,T]\times (0,\infty) \times\R$, which makes $w$ in $\eqref{eqq6}$ continuous in the whole domain. We can apply Theorem~\ref{vertheo} to verify that $w$ coincides with the value function $v$.

\begin{proposition}
\label{remarkonverif}
Assume condition $\eqref{eqq51}$ is satisfied and $g_1\equiv 0$, $g_2$ and $g_3$ are the classical solutions to the equations $\eqref{HJBsplitted}$. Then, function
\begin{equation*}
w(t,a,\epsilon,q)=qae^{\epsilon} g_2(t,\epsilon)+q^2 g_3(t).
\end{equation*}
coincides with the value function $v$ in $\eqref{1valfunc}$ on $[0,T]\times \mathcal{O}$.
\end{proposition}

If condition $\eqref{eqq51}$ is not satisfied, then it is not clear if HJB equation $\eqref{1eqq3tris3}$ admits a classical solution, however, Theorem \ref{thmviscsolgen1} states that the value function $v$ is the unique viscosity solution to the HJB equation $\eqref{1eqq3tris3}$.

\begin{remark}
The model can be extended to cover limit orders as well by introducing a premium for executing limit orders instead of market orders (see Cartea et al.~\cite{Cartea}) and a new state variable $D_t$ as a measure of uncertainty in filling limit orders. We can prove all theorems in this more general setup and, under similar conditions to that in $\eqref{eqq51}$, prove the existence of a classical solution to the HJB equation that coincides with the value function. The model can also be extended to the multi-dimensional case, in which an agent aims to liquidate $m$ different stocks $S^1,\ldots,S^m$ on a basket of $n$ correlated stocks $A^1,\ldots, A^n$.
\end{remark}

\section{FBSDE approach}
\label{sectionFBSDE}
In this section we approach the control problem $\eqref{eqmax}$, by using the stochastic maximum principle (c.f. Pham~\cite[Theorem 6.4.6]{Pham} and Li and Zheng~\cite{LiZheng}). It is a standard approach to write the value function $v$ defined in $\eqref{1valfunc}$ as a solution to an FBSDE and to find the optimal control from the maximization of the Hamiltonian associated to the optimization problem. To apply stochastic maximum principle, we approximate problem $\eqref{eqmax}$, by replacing the stochastic terminal time with a fixed terminal time $T$. We rewrite the value function as
\begin{align}
\label{new1eqq15}
v(t,a,\epsilon,q)= \sup_{c\in\mathcal{A}}\mathbb{E}_t&\bigg[ Q_T \left(S_T-\chi Q_T \right) +\int_t^T c_r(S_r-\eta c_r)\; dr \\
\nonumber
&\quad-\phi_1 \int_t^T Q_r^2 \;dr-\phi_2 \int_t^T S_r Q_r \;dr-\phi_3 \int_t^T A_r Q_r \;dr \bigg].
\end{align}
From SDE $\eqref{SDEdef}$, we get that the only component of the state process $\X_t$ depending on the control $c$ is $Q_t$. We define the Hamiltonian $\mathcal{H}: \mathcal{O} \times [0,\infty)\times \R\to \R$ as
\begin{equation*}
\mathcal{H}(\x,c,y)=-cy+c(ae^{\epsilon}-\eta c)-\phi_1 q^2-\phi_2 ae^\epsilon q-\phi_3 a q.
\end{equation*}
The BSDE associated to our problem is
\begin{equation*}
dY_r=-g(\X_r)dr+\Z_r d\bm{W}_r,
\end{equation*}
with terminal condition $Y_T=h(\X_T)$, where
\begin{equation*}
g(\x)=\frac{\partial}{\partial q} \mathcal{H}(\x)=-2\phi_1 q-\phi_2 ae^\epsilon-\phi_3 a, \qquad h(\x)=ae^{\epsilon}-2\chi q.
\end{equation*}
Using the stochastic maximum principle approach, it follows that the optimal trading strategy $c^*$ coincides with the control function $c$ that maximizes the Hamiltonian $\mathcal{H}$, which is 
\begin{equation}
\label{optimalstrat}
c^*(\x, y)=\frac{1}{2\eta}\left(ae^\epsilon -y \right)^+ .
\end{equation}
In the following we prove that the solution to the above BSDE can be used to find the optimal strategy of the optimization problem $\eqref{new1eqq15}$. The proof of the theorem below immediately follows from Pham~\cite[Theorem 6.4.6]{Pham} using concavity of Hamiltonian $\mathcal{H}$ with respect to variables $(\x,c)$ and maximality of $c^*$ in $\eqref{optimalstrat}$ for the Hamiltonian $\mathcal{H}$.
\begin{theorem}[Stochastic maximum principle]
\label{SMP}
Suppose that the FBSDE
\begin{equation}
\label{BSDEform}
\begin{cases}
dQ_r=-\frac{1}{2\eta}(A_r e^{\varepsilon_r}-Y_r )^+ dr\\
dY_r=-g(A_r, \varepsilon_r, Q_r)dr+\Z_r \cdot d\W_r\\
Q_t=q\\
Y_T=h(A_T, \varepsilon_T, Q_T)
\end{cases}.
\end{equation}
admits a solution $(Q_t^*, Y_t,\Z_t)_{t\in[0,T]}$ and that $(Y_t)_{t\in[0,T]}$ is a progressively measurable, non-negative and square integrable process. Then $c^*$, defined in $\eqref{optimalstrat}$, is the optimal control of problem $\eqref{new1eqq15}$.
\end{theorem}

In the following section we focus on finding a solution to FBSDE $\eqref{BSDEform}$, which is a coupled non-linear Forward-Backward SDE and, to our knowledge, cannot be explicitly solved. In the numerical section below, we use a deep learning-based method, following the one presented in Weinan et al.~\cite{Jentzen}, to find an approximated solution to FBSDE $\eqref{BSDEform}$ and we show that the closed form control in $\eqref{eqq7bis}$ is close to the approximated version of the optimal control in $\eqref{optimalstrat}$.

\section{Numerical Tests}
\label{sectionnumeric}
This section is divided in two parts. The first subsection shows that the closed form control $\eqref{eqq7bis}$ deriving from the HJB equation and the neural network (NN) approximated control $\eqref{optimalstrat}$ deriving from the FBSDE approach are close to each other. In the second subsection we compare the performance of the closed form control $\eqref{eqq7bis}$ based on $A_t$ and $\varepsilon_t$ with respect to the optimal strategy based on two simplified models based on geometric Brownian motion approximations of the liquidating stock price.

\subsection{Neural network approximation vs. closed form control}
\label{subsectionNN}
In this subsection we compare the control obtained through the NN approximated solution of the FBSDE with the closed form control in $\eqref{eqq7bis}$. To numerically find the solution of the FBSDE $\eqref{BSDEform}$ we apply a similar method to the one in Weinan et al.~\cite{Jentzen}. We adapt \cite[Framework 3.2]{Jentzen} to our case by generalizing the implementation to a coupled FBSDE setting with a multidimensional backward equation. The method consists into a neural network approximation of the two solutions $Y$ and $\Z$ of the FBSDE $\eqref{BSDEform}$, where the backward equation is transformed into a forward equation and initial condition $Y_0$ and process $\Z_t$ are chosen in order to minimize the loss
\begin{equation}
\label{lossfunc}
\text{loss}:=\E[|Y_T-g(\X_T)|],
\end{equation}
in order to guarantee the terminal condition $Y_T=g(\X_T)$.

We run several neural network approximations for different model parameters choices and we compare the results of the FBSDE method with the closed form control from Section \ref{sectionclass}. To compare the two methods, we divide the time interval $[0,T]$ in 40 time steps. To calculate the approximated solution of the FBSDE $\eqref{BSDEform}$, we use a 4 layers neural network as in Weinan et al.~\cite{Jentzen} with a batch set made of 64 realizations of $\W$ and a validation set made of 256 realizations. In all numerical examples, we stop training the neural network after 40.000 steps. To calculate the integrals in $\eqref{eqq25}$ used in the representation of the control $c^*$ $\eqref{eqq7bis}$, we apply a quadrature approximation formula. We denote optimal control calculated using NN approximation of the FBSDE solution as $c^{NN}_t$, the inventory process $Q^{NN}_t$ and the wealth process $M^{NN}_t$. For each parameter choice we compare the closed form control $c^*_t$ with $c^{NN}_t$, the inventory process $Q_t$ with $Q^{NN}_t$ and the wealth process $M_t$ with $M^{NN}_t$. 

In the following we show numerical results for 2 different sets of parameters, both satisfying condition $\eqref{eqq51}$. The only differences between the two following settings are volatilities $\sigma_1,\sigma_2$ and the terminal time $T$.
\begin{settings}
\label{sett1}
$A_0=1, \ \epsilon_0=0, \ M_0=1, \ Q_0=20, \ T=0.5, \ \chi=0.5, \ \phi_1=0.003, \ \phi_2=0.06, \ \phi_3=0.06, \ \sigma_1=0.1, \ \sigma_2=0.1, \ k=0.2, \ \eta=0.003, \ \rho=-0.4$.
\end{settings}
\begin{settings}
\label{sett2}
$A_0=1, \ \epsilon_0=0, \ M_0=1, \ Q_0=20, \ T=1, \ \chi=0.5, \ \phi_1=0.003, \ \phi_2=0.06, \ \phi_3=0.06, \ \sigma_1=0.4, \ \sigma_2=0.4, \ k=0.2, \ \eta=0.003, \ \rho=-0.4$.
\end{settings}
In Figure \ref{figconv} is displayed the convergence of loss function $\eqref{lossfunc}$ under Settings \ref{sett1} and \ref{sett2}, which reaches a value lower than $10^{-6}$ after 40.000 training steps in all cases.

\begin{figure}[H]
\centering
  \includegraphics[width=0.7\textwidth]{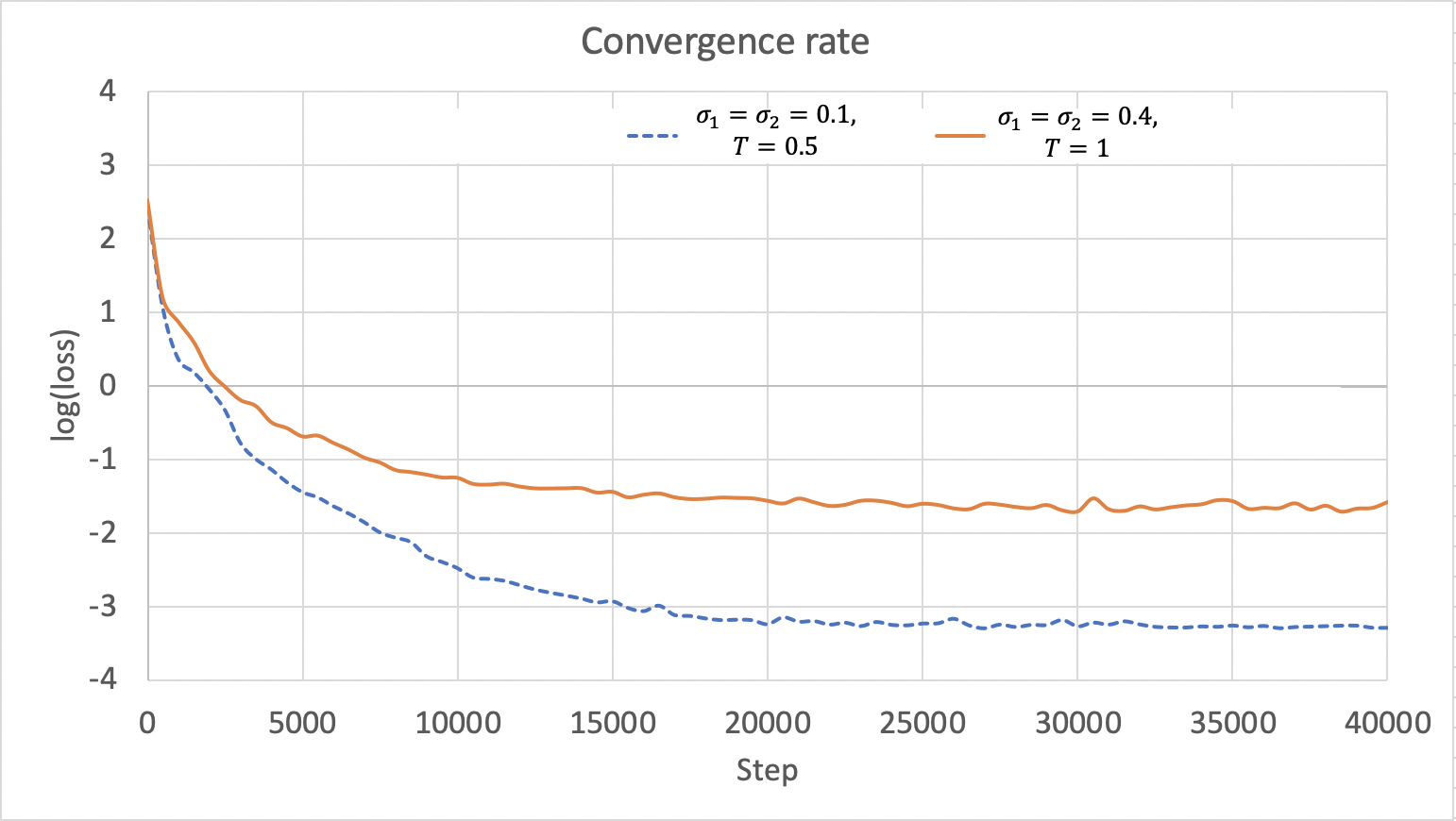}
  \caption{Convergences of logarithm of losses of NNs for three different Settings \ref{sett1} and \ref{sett2}.}
  \label{figconv}
\end{figure}

As a second step, we calculate the average relative discrepancy between $M^{NN}_t$ and $M_t$ over many different realizations of the Brownian motion $(\W_t)_{t\in[0,T]}$ under the two different model parameters sets defined above. In Figure \ref{fig2} is drawn the average and standard deviation of the quantity $\frac{|M^{NN}_t-M_t|}{M^{NN}_t}$ along 400 different realizations of $\W$, for each time step $t$. We notice that in the low volatility case the relative errors $\frac{|M^{NN}_t-M_t|}{M^{NN}_t}$ is low and never exceeding 0.3\%, while in the hight volatility case the discrepancy increases its magnitude to a value of 1\%. 

\begin{figure}[H]
\centering
  \includegraphics[width=0.8\textwidth]{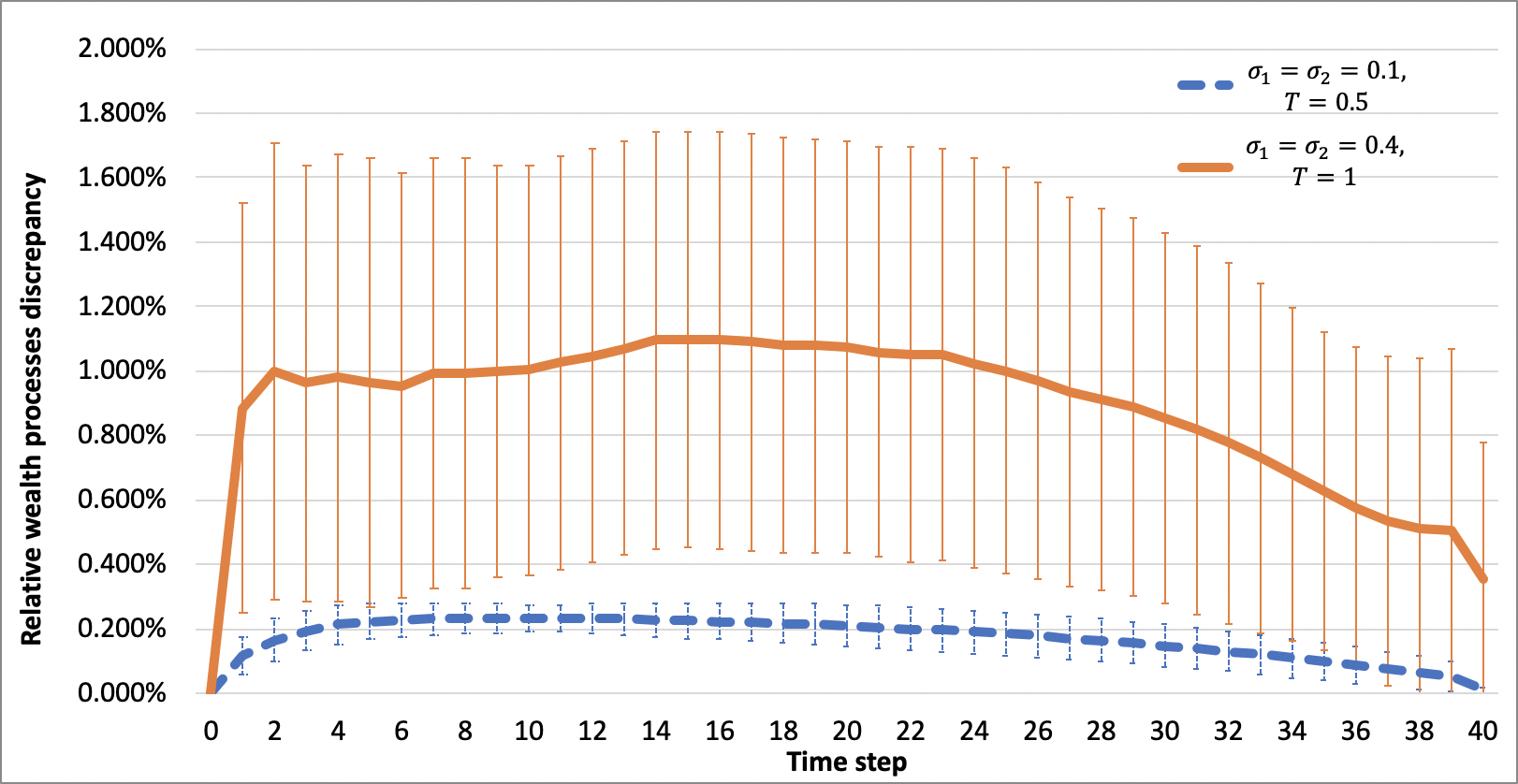}
  \caption{Average and standard deviation of $\frac{|M^{NN}_t-M_t|}{M^{NN}_t}$ along 400 different realizations of $\W$, for each time step $t$ and for different Settings \ref{sett1} and \ref{sett2}.}
  \label{fig2}
\end{figure}

In Table \ref{tableRecap} we group the relative discrepancies in Figure \ref{fig2} and we also consider relative discrepancies of inventories over many different realizations of Brownian motion $(\W_t)_{t\in[0,T]}$ and for different model parameters choices. In Table \ref{tableRecap} is shown the average and standard deviation of the quantities $\frac 1{40}\sum_{t=1}^{40} \frac{|M^{NN}_t-M_t|}{M^{NN}_t}$ and $\frac 1{40}\sum_{t=1}^{40}\frac{|Q^{NN}_t-Q_t|}{Q_0}$ along 400 different realizations of $\W$. We calculate these figures for both Settings \ref{sett1} and \ref{sett2}. 

\restylefloat{table}
\begin{table}[H]
\begin{center}
\begin{tabular}{ |c|c|c|c|c|c|c| }
 \hline
\multirow{2}{*}{Settings} &\multicolumn{2}{|c|}{Av. Rel. Discr. M}&\multicolumn{2}{|c|}{Av. Rel. Discr. Q}&Runtime&Runtime\\ \cline{2-5}
&Mean&St. Dev.&Mean&St. Dev.&closed form&NN \\ \hline
\begin{tabular}{@{}c@{}}$\sigma_1=\sigma_2=0.1,$\\ $T=0.5$\end{tabular} &0.170\% & 0.00057  & 0.105\% & 0.00044& 210 sec.&6850 sec. \\ \hline
\begin{tabular}{@{}c@{}}$\sigma_1=\sigma_2=0.4,$\\ $T=1$\end{tabular} &0.886\% & 0.00601 & 0.665\% & 0.00463&213 sec.& 7120 sec.\\ \hline
\end{tabular}
\caption{Average and standard deviation of $\frac 1{40}\sum_{t=1}^{40} \frac{|M^{NN}_t-M_t|}{M^{NN}_t}$ and $\frac 1{40}\sum_{t=1}^{40}\frac{|Q^{NN}_t-Q_t|}{Q_0}$ along 400 different realizations of $\W$  for Settings \ref{sett1} and \ref{sett2}.}
\label{tableRecap}
\end{center}
\end{table}

In Table \ref{tableInv} we show that the approximation made by removing the stopping time $\tau$ from optimization problem $\eqref{eqmax}$ and fixing it to a terminal time $T$ as in in $\eqref{new1eqq15}$ scarcely affects the value function. Indeed, in Table \ref{tableInv} is shown that the average and standard deviation of the relative discrepancy $|Q_T|/Q_0$ along 400 different realizations of $\W$ are close to $0$. We calculate these figures for both Settings \ref{sett1} and \ref{sett2}. 

\restylefloat{table}
\begin{table}[H]
\begin{center}
\begin{tabular}{ |c|c|c| }
 \hline
Settrings & Mean($|Q_T|/Q_0$)&St. Dev.($|Q_T|/Q_0$)  \\ \hline
\begin{tabular}{@{}c@{}}$\sigma_1=\sigma_2=0.1,$\\ $T=0.5$\end{tabular} &0.0079 & 0.0041 \\ \hline
\begin{tabular}{@{}c@{}}$\sigma_1=\sigma_2=0.4,$\\ $T=1$\end{tabular} &0.0271 & 0.0181 \\ \hline
\end{tabular}
\caption{Average and standard deviation of $|Q_T|/Q_0$ along 400 different realizations of $\W$  for Settings \ref{sett1} and \ref{sett2}.}
\label{tableInv}
\end{center}
\end{table}

In all examples we have shown, the results of the two different methods are close to each other. This increases our confidence in considering the solution of the HJB equation and the trading speed found in Section \ref{sectionclass} respectively equal to the value function and the optimal trading speed of the problem. The computing time necessary  to approximate integrals inside the closed form control representation is around 0.5 seconds for each realization of $\W$. To get an acceptable convergence of the neural network we waited 40.000 steps, taking around 110 minutes for each setting. Once the NN is trained, the computational time for the optimal strategy is around 0.6 seconds for each realization. In conclusion, the NN solution requires a time-consuming initial training that may cause delays any time the model needs to be recalibrated. Once the NN has been trained, the runtimes of the two methods are almost equivalent.

\subsection{Closed form model vs. single stock models}
\label{subsectioncloseform}
In this section we compare our model based on both processes $A_t$ and $\varepsilon_t$ with two simplified models based only on one stock price: one is to approximate the stock price $S$ with a GBM $\tilde{S}$, whose first two moments are equal to those of $S$, and the other is to set the cointegration factor $\varepsilon$ to $0$, whose effect is to approximate the stock price $S$ with that of $A$. Although in both cases the stock price $S$ is approximated with a GBM, the first one is more accurate as it uses the information of the cointegration factor $\varepsilon$. To get the optimal strategy related to approximation $\tilde{S}$, we compare the stock price $S_r$ with a GBM $\tilde{S}_r$ satisfying the following stochastic differential equation (SDE):
\begin{equation*}
d \tilde{S}_r = \tilde{\mu}(r) \tilde{S}_r dr +\tilde{\sigma}(r) \tilde{S}_r dW_r, \quad \tilde{S}_0=s,
\end{equation*}
where $\tilde{\mu}$ and $\tilde{\sigma}$ are deterministic functions that ensure the first two moments of $S_r$ and $\tilde{S}_r$ are the same for $0\le r\le T$, seen at time 0. Since $S_r=A_r e^{\varepsilon_r}$ and $\tilde{S}_r$ are lognormal variables, simple calculus gives
\begin{equation*}
\begin{cases}
\tilde{\mu}(r)=-k\epsilon_0 e^{-kr}+\mu_1+\frac{\sigma_2^2}{2} e^{-2kr}+\rho\sigma_1\sigma_2e^{-kr}\\
\tilde{\sigma}(r)^2 =\sigma_1^2+\sigma_2^2e^{-2kr}+2\rho\sigma_1\sigma_2 e^{-kr}
\end{cases}.
\end{equation*}

\begin{remark}
Note that the initial value of the cointegration factor $\epsilon_0$ appears in $\tilde{\mu}$, which ensures the two processes $S$ and $\tilde{S}$, seen at time $0$, are the same in distribution. They are different, seen at later time $t>0$, as $S$ is determined by two Brownian motions but $\tilde{S}$ by one only. In our numerical test, we approximate the price $S_r$ with the GBM $\tilde{S}_r$ by fixing the cointegration factor to its initial values $\epsilon_0$ throughout the whole trading period $[0,T]$. 
\end{remark}

We solve the stochastic control problem with the same objective function as the one in $\eqref{eqmax}$ without the last term and with $\tilde{S}_r$ instead of $S_r$. The HJB equation is given by
\begin{equation}
\label{HJBGBM}
\frac{\partial w}{\partial t} +\sup_{\tilde{c}\ge 0} \Big[ \frac{\tilde{\sigma}(t)^2}2 \tilde{s}^2 \frac{\partial^2 w}{\partial \tilde{s}^2}+\tilde{\mu}(t) \tilde{s} \frac{\partial w}{\partial \tilde{s}}-\tilde{c}\frac{\partial w}{\partial q}+(\tilde{s}-\eta \tilde{c}) \tilde{c}\Big]-\phi_1 q^2-2\phi_2 q\tilde{s} =0
\end{equation}
on $[0,T)\times (0,\infty)\times [0,q_0]$, with terminal condition $w(T,\tilde{s},q)= q\left(\tilde{s}-\chi q \right)$ and boundary condition $w(t,\tilde{s},0)=0$.
The optimal trading strategy $\tilde{c}^*$ has the following form
\begin{equation*}
\tilde{c}^* =\frac{1}{2\eta} \max\left\{ \tilde{s}-\frac{\partial w}{\partial q}, 0 \right\}.
\end{equation*}
Moreover, equation $\eqref{HJBGBM}$ can be solved using a method similar to the one used in Section 3. Since the solution $w(t,\tilde{s})$ does not depend on $\epsilon$, the equation is easier to be solved.

The second approximation is to use only the price $A$, the optimal trading strategy $c^{A,*}$ has the same formula as that in $\eqref{eqq4}$ with $\epsilon$ equal to $0$.

We compare the performance of our strategy with those of the approximations in different settings. By simulating $S_r$, we can evaluate the performances of the strategies $c^*$, $\tilde{c}^*$ and $c^{A,*}$ respectively based on the price $S$, the GBM price $\tilde{S}$ and the price $A$. To compare the distributions of the cash value $M_\tau+Q_\tau(S_\tau-\chi Q_\tau)$, we run 100 different realizations of process $S_t$ and, by calculating the trading rate for each realization, get the agent's final wealth. We assume that the trader executes orders at equally spaced moment in the interval $[0,T]$. In particular, we consider 100 trades, occurring every $T/100$. The data used for numerical tests are the following: $a_0=6, \ \epsilon_0=0, \ q_0=120, \ T=1, \ \sigma_1=0.3, \ \sigma_2=0.05, \ \mu_1=0, \ \rho=0.5, \ k=0.1, \ \eta=0.01, \ \chi=0.007, \ \phi_1=\phi_2=\phi_3=0.07$. Similar numbers are used in Cartea et al.~\cite{Cartea}. These parameters satisfy condition $\eqref{eqq51}$.

Table \ref{table1} summarizes the key statistics of agent's final wealth using the three different strategies.
\restylefloat{table}
\begin{table}[H]
\begin{center}
\begin{tabular}{ |c|c|c|c|c| } 
 \hline
 Strategy based on & Exp. Val. & St. Dev. & \nth{5} Perc. & \nth{95} Perc. \\
  \hline
Price $S$ & 723.3 & 79.8 & 598.2 & 861.8\\
\hline
&&&&
\\[-1em]
GBM approx. $\tilde{S}$ & 718.8 (-0.6\%) & 95.2 (19.3\%) & 567.3 (-3.5\%)  & 877.3 (1.8\%)\\
 \hline
Stock $A$ & 718.3 (-0.7\%) & 95.8 (20.1\%) & 561.6(-6.1\%) & 877.8 (1.9\%)\\
 \hline
\end{tabular}
\caption{Key statistics of agent's final wealth based on simulations with different optimal strategies. Percentages in brackets represent the discrepancies with respect to the strategy $c^*$ based on stock price $S$.}
\label{table1}
\end{center}
\end{table}

Table \ref{table1} shows that the strategy $c^*$ has the best performance in producing the highest expected value and the lowest standard deviation for agent's final wealth, which indicates using the information of both stocks is highly useful in increasing the final wealth and reducing the risk. The strategy $c^*$ is also the one that guarantees the highest final wealth with 95\% confidence.

Table \ref{table2} summarizes the key statistics of agent's final wealth with change of one parameter while all other parameters are kept the same. In particular, we compare the performance for different correlation coefficient $\rho$, penalty coefficients $\phi_i$, and volatility $\sigma_2$. Table \ref{table2} shows again that strategy using the information of two stocks outperforms those using only one stock.

\begin{table}[ht]
\begin{center}
\begin{tabular}{ |c|c|c|c|c|c| } 
 \hline
Param. Choice & Strategy based on & Exp. Val. & St. Dev. & \nth{5} Perc. & \nth{95} Perc.\\
 \hline
\multirow{3}{*}{$\rho=0$} &Price $S$ & 713.3 & 68.4 & 619.0 & 824.0 \\\cline{2-6}
&&&&&
\\[-1em]
&GBM approx. $\tilde{S}$ & 711.2 & 91.6 & 583.8 & 866.5 \\\cline{2-6}

&Stock $A$ & 710.7 & 95.6 & 579.7 & 883.9 \\
 \hline
\multirow{3}{*}{$\rho=-0.5$} &Price $S$ & 718.2 & 73.6 & 610.8 & 850.4 \\\cline{2-6}
&&&&&
\\[-1em]
&GBM approx. $\tilde{S}$ & 712.7 & 93.2 & 579.3 & 877.1\\\cline{2-6}

&Stock $A$ & 713.0 & 102.8 & 568.5 & 880.4 \\
\hline
\multirow{3}{*}{\shortstack[c]{$\phi_1=\phi_2=\phi_3$\\$=0.05$}} &Price $S$ & 724.9 & 105.4 & 586.4 & 931.9  \\\cline{2-6}
&&&&&
\\[-1em]
&GBM approx. $\tilde{S}$ & 720.9 & 110.3 & 541.2 & 982.9 \\\cline{2-6}

&Stock $A$ & 714.9 & 115.1 & 539.4 & 986.9\\
\hline
\multirow{3}{*}{\shortstack[c]{$\phi_1=\phi_2=\phi_3$\\$=0.09$}} &Price $S$ & 715.7 & 90.8 & 586.0 & 857.7 \\\cline{2-6}
&&&&&
\\[-1em]
&GBM approx. $\tilde{S}$ & 715.1 & 109.3 & 548.9 & 887.0 \\\cline{2-6}

&Stock $A$ & 714.9 & 113.9 & 541.7 & 897.8 \\
\hline
\multirow{3}{*}{$\sigma_2=0.04$} &Price $S$ & 709.0 & 85.0 & 579.6 & 854.9 \\\cline{2-6}
&&&&&
\\[-1em]
&GBM approx. $\tilde{S}$ & 708.0 & 115.2 & 549.8 & 921.7\\\cline{2-6}
&Stock $A$ & 707.5 & 116.2 & 542.7 & 923.7 \\
\hline
\multirow{3}{*}{$\sigma_2=0.06$} &Price $S$ & 729.5 & 87.7 & 599.5 & 881.8  \\\cline{2-6}
&&&&&
\\[-1em]
&GBM approx. $\tilde{S}$ & 726.9 & 112.9 & 564.5 & 928.2  \\\cline{2-6}
&Stock $A$ & 725.7 & 113.7 & 560.5 & 924.6 \\
\hline
\end{tabular}
\caption{Sensitivity analysis to model parameters, by slightly modifying parameters.}
\label{table2}
\end{center}
\end{table}

We also observe that for increasing values of penalty parameters $\phi_i$, we get decreasing expected value and standard deviation of terminal wealth. This is due to the urgency of liquidation introduced by these penalizations, which implies that when trader's risk aversion is higher, the optimal strategy concentrates the liquidation on the initial part of period $[0,T]$, leading to a less volatile but lower expected final wealth.

The opposite behavior can be inferred from different choices of volatility $\sigma_2$. The higher the volatility of cointegration process, the lower the expected final wealth (and the higher the standard deviation).

We also perform a robustness test on three strategies by randomly choosing volatilities $\sigma_1$ and $\sigma_2$ from uniform distributions in which $\sigma_1\in [0.25,0.35]$ and $\sigma_2\in[0.04,0.06]$. We run 300 different simulations of stock price $S$.
\begin{table}[H]
\begin{center}
\begin{tabular}{ |c|c|c|c|c| } 
 \hline
 Strategy based on & Exp. Val. & St. Dev. & \nth{5} Perc. & \nth{95} Perc. \\
  \hline
Price $S$ & 712.7 & 90.6 & 566.2 & 892.2\\
\hline
&&&&
\\[-1em]
GBM approx. $\tilde{S}$ & 711.9 & 117.5 & 558.2 & 982.1 \\
 \hline
Stock $A$ & 709.4 & 125.2 & 545.6 & 991.8\\
 \hline
\end{tabular}
\caption{Robustness test for uniformly randomly chosen values of $\sigma_1\in [0.25,0.35]$ and $\sigma_2\in[0.04,0.06]$}
\label{table3}
\end{center}
\end{table}
Table \ref{table3} shows the conclusions are largely the same as those in Tables \ref{table1} and \ref{table2}.

\section{Conclusion}
\label{sectionconclusion}
We have proved that the value function is the unique viscosity solution of the HJB equation associated to our model, that, under some mild conditions, the HJB equation admits a semi-closed integral representation which makes the calculation for agent's optimal liquidation rate easy and fast. However, the solution of the HJB equation does not always coincide with the value function. We attacked the problem from another perspective, using stochastic maximum principle to solve it. Numerical tests show that the approximate solution of the FBSDE is close to the solution of the HJB equation. This fact increases our confidence in considering the solution of the HJB equation and the trading speed found in Section \ref{sectionclass} respectively equal to the value function and the optimal trading speed of the problem. Numerical tests show that, independent of market conditions, our strategy based on two stock prices outperforms other single stock strategies and approximations with the highest expected final wealth and the lowest standard deviation, is as robust as other strategies known in the literature, based on a single stock.

\appendix
\section{Proofs}
\label{sectionproof}
We first introduce some notations and relations that are used in the proofs. Denote by $\x:=(a, \epsilon, q)$ and 
\begin{align}
\nonumber
\mathcal{G}_1(\x;c)&:= a e^{\epsilon}c-\eta c^2,\\
\label{defGandS2}
\mathcal{G}_2(\x)&:=-\phi_1 q^2-\phi_2 qae^{\epsilon}-\phi_3 qa,\\
\nonumber
\mathcal{S}(\x)&:=q\left(ae^{\epsilon}-\chi q \right).
\end{align}
The objective function $\eqref{1eqq15}$ can be written as 
\begin{equation*}
v^{c}(t,\x):= \mathbb{E}_t\bigg[\int_t^\tau \mathcal{G}_1(\X_r^{t,\x};c_r)dr +\int_t^\tau\mathcal{G}_2(\X_r^{t,\x})dr + \mathcal{S}(\X_\tau^{t,\x})  \bigg].
\end{equation*}
Denote by $\X_r^{t,\x}:=(A_r^{t,\x}, \varepsilon_r^{t,\x}, Q_r^{t,\x})$, the solution of  (\ref{Awritten}), (\ref{epsilonwritten}), and (\ref{qwritten}) with the initial condition $\X_t=\x\in \mathcal{O}$ and square integrable feasible control  $c\in \mathcal{A}$ and $t\in[0,T]$. We omit the superscript in $X^{t,\x}$ and we denote it as $X$ when the initial conditions are clear from the context.
Using the standard stochastic analysis for stochastic differential equations,   we have (see, e.g., Krylov~\cite{Kry}) 
\begin{equation}
\label{eqq3.5gen}
\lim_{h\searrow 0^+} \E_t \left[ \sup_{r\in[t,t+h]} \left|\X_r^{t,\x}-\mathbf{x} \right|^2\right]=0.
\end{equation}

\begin{lemma}
\label{lemmaonboundexpeps}
Let $p>0$ be a constant and  $t\in[0,T]$, then there exists a constant $C>0$ such that for any $\x\in\mathcal{O}$
\begin{equation}
\label{eqqboundexp1}
\E_t\left[ \sup_{r\in[t,T]} e^{p\varepsilon_r^{t,\x}}\right]\le Ce^{p|\epsilon|}
\end{equation}
and for any $r\in[t,T]$
\begin{equation}
\label{eqqboundexp2}
\E_t\left[e^{ p\left(\varepsilon_{r}^{t,\x}-\epsilon\right)}\right]\le Ce^{p|\epsilon|\left(1-e^{-k(r-t)}\right)} .
\end{equation}
\end{lemma}
\begin{proof}
Process $\varepsilon_r$, defined in $\eqref{epsilonwritten}$, has explicit formulation
$$\varepsilon_r^{t,\x} =\epsilon e^{-k(r-t)}+\sigma_2 \int_t^r e^{-k(r-s)}\left(\rho dW^1_s+\sqrt{1-\rho^2} dW^2_s\right)=:\epsilon e^{-k(r-t)}+\sigma_2 e^{-kr} M_r^t.$$
Using Ito's formula, we have the process $N_r^t:=e^{p\sigma_2 M_r^t}$ satisfies the following SDE
$$
dN_r^t=\frac{p^2 \sigma_2^2}{2} e^{2kr} N_r^t dr + p\sigma_2 e^{kr} N_r^t \left(\rho dW^1_r+\sqrt{1-\rho^2} dW^2_r\right), \quad N_t^t=1. $$
The  SDE above satisfies conditions of Krylov \cite[Corollary 2.6.12]{Kry}, then there exists a constant $C>0$ such that
$\E\left[ \sup_{r\in[t,T]} e^{p\sigma_2M_r^t}\right]\le \E\left[ \sup_{r\in[t,T]} |N_r^t|\right]\le \E\left[ \sup_{r\in[0,T]} |N_r^0|\right]\le C$.
Using that for any $\kappa \in[0,1]$, $e^{\kappa x}\le e^x+1$ on $\R$, we get $\eqref{eqqboundexp1}$:
$$\E_t\left[ \sup_{r\in[t,T]} e^{p \varepsilon_r^{t,\x}}\right]\le \E_t\left[ \sup_{r\in[t,T]} e^{p\epsilon e^{-k(r-t)}}\left(e^{p\sigma_2 M_r^t}+1\right)\right]\le e^{p|\epsilon|}\left(\E_t\left[ \sup_{r\in[t,T]} e^{p \sigma_2M_r^t}\right]+1\right)\le (C+1)e^{p|\epsilon|} .$$
Finally, we apply similar arguments to get $\eqref{eqqboundexp2}$:
$$\E_t\left[e^{p\left(\varepsilon_{r}^{t,\x}-\epsilon\right)}\right]\le \E_t\left[ e^{p\epsilon \left(e^{-k(r-t)}-1\right)}\left(e^{p\sigma_2 M_r^t}+1\right)\right]\le (C+1)e^{p|\epsilon|\left(1-e^{-k(r-t)}\right)} .$$
\end{proof}

\begin{lemma}
\label{lemmaboundsupS}
Let $t\in[0,T]$ and $\mathbf{x}, \mathbf{x'}\in \mathcal{O}$. There exists $C_p>0$, independent of $t$, such that
\begin{equation}
\label{boundSexp}
\E_t\left[\sup_{r\in[t,T]}\left|S_{r}^{t,\x}-S_{r}^{t,\x'}\right|^p \right] \le C_p e^{p|\epsilon'|} \left(|a-a'|^p+a^p\left|e^{\epsilon-\epsilon'}-1\right|^p\right).
\end{equation}
\end{lemma}
\begin{proof}
Using results on GBM (cf. \cite[Theorem 1.3.15]{Pham} and \cite[Corollary 2.6.12]{Kry}) and using $\eqref{eqqboundexp1}$, we have that, for any $p\ge 1$, there exists a constant $C_p>0$ independent of $t$ and $\x$ such that 
\begin{equation}
\label{eqqGBMbound}
\E_t\left[\sup_{r\in[t,T]} \left(A_r^{t,\x}\right)^p \right]<C_p a^p, \qquad \E_t\left[\sup_{r\in[t,T]} e^{p\varepsilon_r^{t,\x}}\right]<C_p e^{p |\epsilon|}.
\end{equation}
Hence, for any $p\ge 1$ there exists $C_p>0$, independent of $t$, so that
\begin{align*}
\E_t\left[\sup_{r\in[t,T]}\left|S_{r}^{t,\x}-S_{r}^{t,\x'}\right|^p\right] &\le 2^p\left(\E_t\left[\sup_{r\in[t,T]} \left(A_{r}^{t,\x}\right)^{2p}\right]\right)^{\frac 12}\left(\E_t\left[\sup_{r\in[t,T]}\left|e^{\varepsilon_{r}^{t,\x}}-e^{\varepsilon_{r}^{t,\x'}}\right|^{2p}\right]\right)^{\frac 12}\\
&\qquad+2^p\left(\E_t\left[\sup_{r\in[t,T]} e^{2p\varepsilon_{r}^{t,\x'}}\right]\right)^{\frac 12}\left(\E_t\left[\sup_{r\in[t,T]}\left|A_{r}^{t,\x}-A_{r}^{t,\x'}\right|^{2p}\right]\right)^{\frac 12}\\
& \le C_p\left(a^p\left(\E_t\left[\sup_{r\in[t,T]}\left|e^{\varepsilon_{r}^{t,\x}}-e^{\varepsilon_{r}^{t,\x'}}\right|^{2p}\right]\right)^{\frac 12}+e^{p|\epsilon'|}\left(\E_t\left[\sup_{r\in[t,T]}\left|A_{r}^{t,\x}-A_{r}^{t,\x'}\right|^{2p}\right]\right)^{\frac 12}\right).
\end{align*}
The explicit formulations for processes $A_r$ and $\varepsilon_r$ give
\begin{align}
\label{boundonlipsA}
\left|A_{r}^{t,\x}-A_{r}^{t,\x'}\right|&= |a-a'|e^{\left(\mu_1-\frac{\sigma_1^2}2\right)(r-t)+\sigma_1W^1_{r-t}},\\
\nonumber
\left|e^{\varepsilon_{r}^{t,\x}}-e^{\varepsilon_{r}^{t,\x'}}\right|&=e^{\varepsilon_{r}^{t,\x'}}\left| e^{(\epsilon-\epsilon')e^{-k(r-t)}} -1\right|\le e^{\varepsilon_{r}^{t,\x'}}\left| e^{\epsilon-\epsilon'} -1\right|.
\end{align}
Here we have used the fact that for $\kappa\in[0,1]$, $\left|e^{\kappa x}-1\right|\le |e^x -1|$ on $\R$. Using similar argument as in $\eqref{eqqGBMbound}$, we get that there exists $C_p>0$ such that
\begin{align*}
\left(\E_t\left[\sup_{r\in[t,T]}\left|A_{r}^{t,\x}-A_{r}^{t,\x'}\right|^{2p}\right]\right)^{\frac 12}&\le C_p|a-a'|^p,\\
\left(\E_t\left[\sup_{r\in[t,T]}\left|e^{\varepsilon_{r}^{t,\x}}-e^{\varepsilon_{r}^{t,\x'}}\right|^{2p}\right]\right)^{\frac 12}&\le C_pe^{p|\epsilon'|}\left| e^{\epsilon-\epsilon'} -1\right|^p.
\end{align*}
\end{proof}

\subsection{Proof of Theorem \ref{vertheo}}
We follow the proof of Verification Theorem Pham~\cite[Theorem 3.5.2]{Pham} to show that $v$ and $w$ coincide on $[0,T]\times \mathcal{O}$. The two main differences between our setting and Pham's setting are that solution $w$ does not satisfy a quadratic growth condition and the presence of a stopping time $\tau$ in the definition of value function in our case. We define sequence of stopping time $\tau_n$ similarly as in \cite[proof of Theorem 3.5.2]{Pham}, by capping it with the stopping time $\tau$:
$$\tau_n:=\tau \wedge\inf_{s\ge t}\left\{ \int_t^s |\nabla_{\x}w(r,\X_r)'\sigma(\X_r)|^2\;dr\ge n \right\}.$$
We notice that $\tau_n\nearrow \tau$ and the stopped process $(\int_t^{s\wedge \tau_n} \nabla_{\x}w(r,\X_r)'\sigma(\X_r)\;dr)_{s\in[t,T]}$ is a martingale. Let $c\in \mathcal{A}$ be fixed. By taking the expectation of the Ito's representation of $w(s, \X_{s})$, we get
\begin{align}
\label{eqqq20}
\E_t\left[w(s\wedge\tau_n, \X_{s\wedge \tau_n})\right]=w(t,\mathbf{x})+\E_t&\bigg[\int_t^{s\wedge \tau_n} \bigg(\frac{\partial w}{\partial t}(r,\X_r)+\mathcal{L}w(r,\X_r) -c\frac{\partial w}{\partial q}(r,\X_r)  \bigg) dr \bigg].
\end{align}
Since $w$ is a solution to the HJB equation $\eqref{1eqq3tris3}$, for a general $c\in \mathcal{A}$
\begin{equation*}
\frac{\partial w}{\partial t}(r,\X_r) +\mathcal{L}w(r,\X_r)+\mathcal{G}_1(\X_r;c_r) +\mathcal{G}_2(\X_r) -c_r\frac{\partial w}{\partial q}(r,\X_r) \le 0
\end{equation*}
and applying it to $\eqref{eqqq20}$, we get
\begin{align}
\label{eqqq19}
&\E_t\left[w(s\wedge\tau_n, \X_{s\wedge \tau_n})\right]\le w(t,\mathbf{x})-\E_t\left[\int_t^{s\wedge \tau_n} \left(\mathcal{G}_1(\X_r;c_r)  +\mathcal{G}_2(\X_r)\right) dr \right].
\end{align}
We apply the dominated convergence theorem to previous inequality. Both sides are bounded by an integrable process independent on $n$. By using boundedness of process $(Q_r)_{r\in[0,T]}$ and Holder inequality, we have that
\begin{align*}
&\E_t\left[\left|\int_t^{s\wedge \tau_n} \left(\mathcal{G}_1(\X_r;c_r) +\mathcal{G}_2(\X_r)\right) dr \right| \right] \le \left(\E_t\left[\int_t^{T} A_r^2 e^{2\varepsilon_r}\; dr \right]\right)^{\frac 12} \left( \E_t\left[\int_t^{T} c_r^2\; dr \right]\right)^{\frac 12} \\
&\qquad\qquad+\eta\E_t\left[\int_t^{T} c_r^2\; dr \right]+\phi_1 q^2+\phi_2 q\E_t\left[\int_t^{T} A_r e^{\varepsilon_r}\; dr \right] +\phi_3 q \E_t\left[\int_t^{T} A_r\; dr \right],
\end{align*}
which is bounded independently on $n$, using $\eqref{eqqGBMbound}$ and square integrability of control process $(c_r)_{r\in[t,T]}$. By $\eqref{eqqGBMbound}$, Holder inequality, growth condition on $w$ and recalling that $Q_r$ is bounded, we conclude that for any $(t,\x)\in[0,T]\times \mathcal{O}$, there exists $C>0$ independent of $n$ such that 
$$\E_t\left[w(s\wedge \tau_n, \X_{s\wedge \tau_n})\right]\le C(1+q^2)(1+a^{p_1})(1+e^{p_2\epsilon}),$$
We apply the dominated convergence theorem to $\eqref{eqqq19}$ by sending $n\to \infty$:
\begin{align*}
\E_t\left[w(s\wedge \tau, \X_{s\wedge \tau})\right]\le w(t,\mathbf{x})-\E_t\left[\int_t^{s\wedge \tau} \left(\mathcal{G}_1(\X_r;c_r)  +\mathcal{G}_2(\X_r)\right) dr \right].
\end{align*}
Since $w$ is continuous on $[0,T]\times \mathcal{O}$, by sending $s$ to $T$, we obtain by the dominated convergence theorem:
\begin{equation*}
\E_t\left[w(\tau, \X_{ \tau})\right]\le w(t,\mathbf{x})-\E_t\left[\int_t^{\tau} \left(\mathcal{G}_1(\X_r;c_r)  +\mathcal{G}_2(\X_r)\right) dr \right].
\end{equation*}
By terminal and boundary condition of HJB equation $\eqref{1eqq3tris3}$, we know that $w(\tau,\mathbf{x})=\mathcal{S}(\x)$, so we have
\begin{align*}
\E_t\left[\mathcal{S}(\X_{ \tau})\right]\le w(t,\mathbf{x})-\E_t\left[\int_t^{\tau} \left(\mathcal{G}_1(\X_r;c_r)  +\mathcal{G}_2(\X_r)\right) dr \right],
\end{align*}
which implies that
\begin{align*}
w(t,\mathbf{x})\ge \E_t\left[\mathcal{S}(\X_{ \tau})+\int_t^{\tau} \left(\mathcal{G}_1(\X_r;c_r)  +\mathcal{G}_2(\X_r)\right) dr \right]=v^{c}(t,\x).
\end{align*}
From arbitrariness of $c\in\mathcal{A}$, it follows that $v\le w$ on $[0,T]\times \mathcal{O}$.

To prove that $v\ge w$ on $[0,T]\times \mathcal{O}$, we proceed as before, by getting a similar version of $\eqref{eqqq20}$ in which the control process $c_r$ is substituted by the optimal control $c^*(r,\X_r)$:
\begin{equation*}
\E_t\left[w(s\wedge\tau_n, \X_{s\wedge \tau_n})\right]=w(t,\mathbf{x})+\E_t\left[\int_t^{s\wedge \tau_n} \left(\frac{\partial w}{\partial t}(r,\X_r) +\mathcal{L}w(r,\X_r)-c^*(r,\X_r)\frac{\partial w}{\partial q}(r,\X_r)\right) dr \right].
\end{equation*}
By applying optimality of $c^*$, we get
\begin{align*}
\E_t\left[w(s\wedge\tau_n, \X_{s\wedge \tau_n})\right]=w(t,\mathbf{x})+\E_t\left[\int_t^{s\wedge \tau_n} \left(\mathcal{G}_1(\X_r;c^*(r,\X_r))+\mathcal{G}_2(\X_r)\right) dr \right].
\end{align*}
Proceeding as before, we apply dominated convergence theorem to both sides of previous expression. By sending $n\to \infty$ and then sending $s$ to $T$, we get
\begin{align*}
\E_t\left[w(\tau, \X_{\tau})\right]=w(t,\mathbf{x})-\E_t\left[\int_t^{\tau} \left(\mathcal{G}_1(\X_r;c^*(r,\X_r))+\mathcal{G}_2(\X_r)\right) dr \right].
\end{align*}
By terminal condition of the HJB equation, $w(\tau,\mathbf{x})=\mathcal{S}(\mathbf{x})$, so we have
\begin{align*}
&w(t,\mathbf{x})=\E_t\bigg[\mathcal{S}(\X_{\tau})+\int_t^{\tau} \left(\mathcal{G}_1(\X_r;c^*(r,\X_r))+\mathcal{G}_2(\X_r)\right) dr \bigg]=v^{c^*}(t,\mathbf{x}).
\end{align*}
This shows that $w(t,\mathbf{x})=v^{c^*}(t,\mathbf{x})\le v(t,\mathbf{x})$ on $[0,T]\times \mathcal{O}$.

\subsection{Proof of Theorem \ref{thmviscsolgen1}}
To prove the result, we first give a technical lemma.
\begin{lemma}[Comparison Principle]
\label{thmcompgen}
Let $U$ (respectively $V$) be an upper semicontinuous viscosity subsolution (resp. lower semicontinuous viscosity supersolution) to the following HJB equation
\begin{equation}
\label{eqq3tris5gen}
-\frac{\partial v}{\partial t}(t,\mathbf{x})-\mathcal{L} v(t,\mathbf{x}) -\sup_{c\ge 0} \left[-c \frac{\partial v}{\partial q}(t,\mathbf{x})+\mathcal{G}_1(\x;c) \right] -\mathcal{G}_2(\x)=0
\end{equation}
for any $(t,\mathbf{x})\in[0,T)\times (0,\infty)\times\R\times (0,\bar{Q}_0)$. Assume there exist $C,\kappa>0$ and $m\in \N$ such that
\begin{equation}
\label{growthcond2gen}
|U(t,\x)|+|V(t,\x)|\le C\left( 1+a^{m} \right)\left(1+e^{\kappa |\epsilon|}\right)
\end{equation}
for any $(t,\x)\in[0,T]\times \mathcal{O}$. If 
\begin{equation}
\label{condtermandbound}
U(T,\cdot)\le V(T,\cdot) \text{ on } \mathcal{O} \text{ and } U(t,a,\epsilon,0)\le V(t,a,\epsilon,0) \text{ for any } (t,a,\epsilon)\in[0,T]\times (0,\infty)\times  \R,
\end{equation}
then $U\le V$ on $[0,T]\times \mathcal{O}$.
\end{lemma}
\begin{proof}
The Lemma is proved following Pham~\cite[proof of Theorem 4.4.5]{Pham}. The main difference between our statement and Pham~\cite[Theorem 4.4.5]{Pham} is our functions $U$ and $V$ are not polynomially growing and are defined in a subset of $\R^n$ space. We apply the first step in \cite[proof of Theorem 4.4.3]{Pham}, which provides an equivalent formulation for the HJB equation $\eqref{eqq3tris5gen}$. Let $\beta>0$ be specified later, $\bar{U}(t,\x)=e^{\beta t}U(t,\x)$ and $\bar{V}(t,\x)=e^{\beta t}V(t,\x)$, then $\bar{U}$ and $\bar{V}$ are respectively subsolution and supersolution to
\begin{equation}
\label{newHJB}
-\frac{\partial w}{\partial t}(t,\mathbf{x})+\beta w(t,\mathbf{x})-\mathcal{L}w(t,\mathbf{x}) -\sup_{c\ge 0} \left[-c \frac{\partial w}{\partial q}(t,\mathbf{x})+e^{\beta t}\mathcal{G}_1(\x;c) \right] -e^{\beta t}\mathcal{G}_2(\x)=0
\end{equation}
for any $(t,\mathbf{x})\in[0,T)\times \mathcal{O}$. With a slight abuse of notation, in the remaining of the proof, we denote $\bar{U}$, $\bar{V}$ respectively $U$, $V$ and we replace equation $\eqref{eqq3tris5gen}$ with $\eqref{newHJB}$.

We adapt second step of \cite[proof of Theorem 4.4.3]{Pham} to show that there exists a function $\varphi (t,\x)$ such that for any $\delta>0$, $V+\delta \varphi$ is a supersolution to $\eqref{newHJB}$. Define $p(a)=C\left( 1+a^d \right)$, where $d>\max(m,2)$ and $C, m$ are as in $\eqref{growthcond2gen}$. Define for any $\mathbf{x} \in \mathcal{O}$
\begin{align*}
\varphi(\mathbf{x})&=\frac{1}{a^2}-\ln\left(\frac{\bar{Q}_0-q}{\bar{Q}_0+1}\right)+\left(1+p(a)^2\right)\left(1+e^{b \epsilon}+e^{-b \epsilon}\right)
\end{align*}
where $b>\max(\kappa, 2)$ and $\kappa$ is defined in $\eqref{growthcond2gen}$. We observe that $\varphi(\mathbf{x})$ is nonnegative and infinitely many times differentiable on $\mathcal{O}$. An explicit calculation shows that
\begin{align*}
&-\frac{\partial \varphi}{\partial t}(\mathbf{x})+\beta \varphi(\mathbf{x}) -\mathcal{L}\varphi(\mathbf{x})-\sup_{c\ge 0} \left[ -c \frac{\partial \varphi}{\partial q}(\mathbf{x}) \right]\\
&\qquad= \beta\varphi +\frac{2\mu_1-3\sigma_1^2}{a^2}-b\left(e^{b \epsilon}-e^{-b \epsilon}\right)\left(-k\epsilon \left(1+p^2\right)+2p p' \rho\sigma_1\sigma_2 a \right)-\frac{\sigma_2^2}2 b^2 \left(1+p^2\right) \left(e^{b \epsilon}+e^{-b \epsilon}\right) \\
&\qquad\qquad-\left(\sigma_1^2 a^2\left(\left(p'\right)^2 +p p'' \right)+2\mu_1 a pp'\right)\left(1+e^{b \epsilon}+e^{-b \epsilon}\right)-\sup_{c\ge 0} \left[- \frac{c}{\bar{Q}_0-q}\right].
\end{align*}
We observe that $\sup_{c\ge 0} \left[- \frac{c}{\bar{Q}_0-q}\right]$ and there exists a constant $C_1>0$ such that $
a^2\left(\left(p'(a)\right)^2+p(a)p''(a)\right)\le C_1 p(a)^2$ and $
2ap(a)p'(a)\le C_1 p(a)^2$ for any $a\ge 0$. Simple calculus shows that $2\left(e^{b \epsilon}-e^{-b \epsilon}\right)p p' a\le 2\left(e^{b \epsilon}+e^{-b \epsilon}\right)p p' a\le C_1\left(e^{b \epsilon}+e^{-b \epsilon}\right)p^2 \le C_1 \varphi$. Then, we get
\begin{align}
\nonumber
&-\frac{\partial \varphi}{\partial t}(\mathbf{x})+\beta \varphi(\mathbf{x})-\mathcal{L}\varphi(\mathbf{x}) -\sup_{c\ge 0} \left[ -c \frac{\partial \varphi}{\partial q}(\mathbf{x})\right]\\
\nonumber
&\quad\ge \left(\beta-C_1\rho\sigma_1\sigma_2b\right)\varphi +\underbrace{k\epsilon b\left(e^{b \epsilon}-e^{-b \epsilon}\right)\left(1+p^2\right)}_{\ge 0} \\
\nonumber
&\qquad -\bigg[ \frac{\sigma_2^2}2 b^2 \left(1+p^2\right)\left(e^{b \epsilon}+e^{-b \epsilon}\right)+\frac{3\sigma_1^2}{a^2}+C_1\left( \sigma_1^2+\mu_1\right)p^2\left(1+e^{b \epsilon}+e^{-b \epsilon}\right)\bigg]\\
\label{eqset1}
&\quad\ge \left(\beta  -C_1\left(\sigma_1^2+\mu_1+\rho\sigma_1\sigma_2b\right)-\frac{\sigma_2^2}2 b^2-3\sigma_1^2\right)\varphi.
\end{align}
Choosing $\beta>0$ so that $\beta>C_1\left(\sigma_1^2+\mu_1+\rho\sigma_1\sigma_2b\right)+\frac{\sigma_2^2}2 b^2+3\sigma_1^2$, we get that for any $\delta>0$, the function $V_\delta =V+\delta \varphi$ is, as $V$, a supersolution to $\eqref{newHJB}$. Moreover, from definition of $\varphi$, and from growth conditions on $U$, $V$ we have that for $\epsilon\to \pm \infty$ and $a\to +\infty$, $\varphi$ grows more rapidly than $U$ and $V$. For $a\to 0$ and $q\to \bar{Q}_0$, $U$ and $V$ are finite, while $\varphi\to +\infty$. This implies that for any $\delta>0$, there exists an open and bounded set $\mathcal{O}_\delta$ so that $\bar{\mathcal{O}_\delta}\subset \mathcal{O}$ and
\begin{equation}
\label{eqqlimitbound}
\sup_{(t,\mathbf{x})\in [0,T]\times \mathcal{O}} (U-V_\delta)(t,\mathbf{x})=\max_{(t,\mathbf{x})\in [0,T]\times\left\{\mathbf{x}\in \R^3| \mathbf{x} \in \mathcal{O}_\delta \text{ or } q=0\right\}} (U-V_\delta)(t,\mathbf{x}) .
\end{equation}
To conclude the proof of the Lemma, we need to show that 
\begin{equation}
\label{toproveinlemma}
\forall \delta>0, \ \sup_{(t,\mathbf{x})\in [0,T]\times \mathcal{O}} (U-V_\delta)(t,\mathbf{x})\le 0.
\end{equation}
However, using $\eqref{condtermandbound}$, upper semicontinuity of $U$, lower semicontinuity of $V$ and that $\varphi(\cdot,\cdot)\ge 1$, we get that 
\begin{align}
\label{eqq63}
&\forall \delta>0 \text{ exists } \gamma>0 \text{ s.t. } (U-V_\delta)(t,\x)<0 \text{ when } t\in (T-\gamma,T] \text{ and}\\
\label{eqq63bis}
&\forall \delta>0 \text{ exists } \gamma>0 \text{ s.t. } (U-V_\delta)(t,\x)<0 \text{ when } q<\gamma.
\end{align}
By applying $\eqref{eqqlimitbound}$, $\eqref{eqq63}$ and $\eqref{eqq63bis}$ we reduce our objective from $\eqref{toproveinlemma}$ to the proof of
\begin{equation}
\label{toproveinlemmabis}
\forall \delta>0, \ M_\delta:=\max_{(t,\mathbf{x})\in [0,T)\times \mathcal{O}_\delta} (U-V_\delta)(t,\mathbf{x})\le 0.
\end{equation}
To prove the above statement, we assume by contradiction that $M_\delta>0$. On the bounded set $\mathcal{O}_\delta$, functions $\mu$ and $\sigma$ are uniformly Lipschitz and $\mathcal{G}_2$ is uniformly continuous. Then, by following Pham~\cite[proof of Theorem 4.4.5]{Pham}, we get that for any $\delta>0$, $\beta M_\delta\le 0$, which is a contradiction. We conclude that for any $\delta>0$, $M_\delta\le 0$ and so both $\eqref{toproveinlemmabis}$ and $\eqref{toproveinlemma}$ hold true. By taking limit of $\delta$ going to $0$ in $\eqref{toproveinlemma}$, we get that $(U-V)(\cdot,\cdot)\le 0$ in $[0,T]\times \mathcal{O}$, which concludes the proof.
\end{proof}

We now prove Theorem \ref{thmviscsolgen1}. By analysing value function $v$, we get the following upper and lower bounds. Using boundedness of process $(Q_r)_{r\in[t,T]}$ and $\eqref{eqqGBMbound}$, there exists $C>0$ such that for any $(t,\x)\in[0,T]\times \mathcal{O}$,
\begin{align*}
v(t,\x)&\le \sup_{c\in\mathcal{A}}\E_t[M_\tau]+\sup_{c\in\mathcal{A}}\E_t[Q_\tau S_\tau]-\inf_{c\in\mathcal{A}}\E_t\left[\chi Q_\tau^2+\phi_1 \int_t^\tau Q_r^2 \;dr+\phi_2 \int_t^\tau S_r Q_r \;dr+\phi_3 \int_t^\tau A_r Q_r \;dr\right]\\
&\le \sup_{c\in\mathcal{A}}\E_t\left[\int_t^\tau c_r (S_r-\eta c_r)\; dr\right]+q\E_t\left[\sup_{r\in[t,T]} S_r\right]\le \frac T{4\eta}\E_t\left[\sup_{r\in[t,T]} S_r^2\right]+q\E_t\left[\sup_{r\in[t,T]} S_r\right]\\
&\le C (a+a^2) e^{C|\epsilon|}.
\end{align*}
On the other hand, by choosing $c\equiv 0$, there exists $C>0$ such that for any $(t,\x)\in[0,T]\times \mathcal{O}$
\begin{align*}
v(t,\x)&\ge v^0(t,\x)= \E_t[q S_T]-\E_t\left[\chi q^2+\phi_1 \int_t^T q^2 \;dr+\phi_2 \int_t^T S_r q\;dr+\phi_3 \int_t^T A_r q\;dr\right]\\
&\ge -\E_t\left[\chi q^2+\phi_1 Tq^2+\phi_2 Tq \sup_{r\in[t,T]} S_r+\phi_3T q \sup_{r\in[t,T]} A_r\right]\\
&\ge - C(1+a)\left(1+ e^{C|\epsilon|}\right).
\end{align*}
Here in the last inequality we have used $\eqref{eqqGBMbound}$. All conditions in \cite[Propositions 4.3.1 and 4.3.2]{Pham} are satisfied. In particular, \cite[Condition (3.5)]{Pham} holds true in $\eqref{eqq3.5gen}$ and $v$ is locally vounded as proved in upper and lower bounds above. Then, by applying Pham~\cite[Propositions 4.3.1 and 4.3.2]{Pham}, we prove that the value function $v$ is a viscosity solution to the HJB equation $\eqref{1eqq3tris3}$. \\
Using the above upper and lower bounds we get that $v$ satisfies the growth condition $\eqref{growthcond2gen}$. Then, using Comparison Principle Lemma \ref{thmcompgen}, we conclude that value function $v$ is the unique viscosity solution of HJB equation $\eqref{1eqq3tris3}$. \qed

\subsection{Proof of Proposition \ref{propAtilde}}
To prove the result, we first give one technical lemma.
\begin{lemma}
\label{lemmalocbound}
Let $\gamma>0$ be fixed and let $\tilde{\mathcal{A}}_{\gamma}$, defined in $\eqref{defAtilde}$, be the set of admissible controls. Then, the value function $v$, defined in $\eqref{1valfunc}$, has the following property:
\begin{align}
\label{boundloc}
|v(t,\x)-v(t,\x')|\le C\left(|a-a'|+\left|e^{\epsilon-\epsilon'}-1\right|+|q-q'|^{\frac{\gamma}{\gamma+1}}\right)\left(1+a+a'\right)^C\left(1+e^{C|\epsilon|}+e^{C|\epsilon'|}\right)\left(1+q+q'\right)^C
\end{align}
for any $t\in[0,T]$ and $\x,\x'\in \mathcal{O}$, where $C>0$ is a constant independent of $t$.
\end{lemma}
\begin{proof}
Let $t\in[0,T]$ be fixed and $\x, \x'\in \mathcal{O}$. We assume w.l.o.g. that $q\ge q'$. Denote $\X_r^{t,\x}$ and $\X_r^{t,\x'}$ the two solutions to $\eqref{SDEdef}$ with initial conditions $(t,\x)$ and $(t,\x')$ respectively and the stopping times $\tau=T\wedge \min\{r\ge t\; | \; Q_r^{t,\x}=0\}$ and $\tau'=T\wedge \min\{r\ge t\; | \; Q_r^{t,\x'}=0\}$. We observe that
\begin{align}
\nonumber
|v(t,\x)-v(t,\x')|&\le \sup_{c\in \tilde{\mathcal{A}}_{\gamma}(t,\x)\cap \tilde{\mathcal{A}}_{\gamma}(t,\x')} \mathbb{E}_t\Bigg[ \left|M_\tau^{t,\x}-M_{\tau'}^{t,\x'}\right|+ \left|Q_\tau^{t,\x}\left(S_\tau^{t,\x}-\chi Q_\tau^{t,\x} \right)-Q_{\tau'}^{t,\x'}\left(S_{\tau'}^{t,\x'}-\chi Q_{\tau'}^{t,\x'} \right)\right|\\
\label{eqqd1}
&\qquad\qquad\qquad+\left| \int_t^\tau \mathcal{G}_2(\X_r^{t,\x}) \;dr-\int_t^{\tau'} \mathcal{G}_2(\X_r^{t,\x'}) \;dr \right|\Bigg].
\end{align}
We fix a control $c\in \tilde{\mathcal{A}}_{\gamma}(t,\x)\cap \tilde{\mathcal{A}}_{\gamma}(t,\x')$. We observe that $\tau\ge \tau'$ $\Pro$-a.s., since $Q_r^{t,\x}\ge Q_r^{t,\x'}$ $\Pro$-a.s. for any $r\ge t$, by the assumption $q\ge q'$. Recall that
\begin{equation}
\label{boundQind}
\forall \omega \in \{\tau'<T\}, \ \forall r\ge \tau' (\omega), \quad Q_r^{t,\x}(\omega)\le q-q',
\end{equation}
we get that
\begin{align}
\nonumber
\E_t\left[|Q_{\tau}^{t,\x}-Q_{\tau'}^{t,\x'}|\right]&\le\E_t\left[\mathds{1}_{\tau<T, \tau'<T}\cdot 0+ \mathds{1}_{\tau'<\tau=T}Q_{T}^{t,\x}+ \mathds{1}_{\tau=\tau'=T}\left|Q_{T}^{t,\x}-Q_{T}^{t,\x'}\right|\right]\\
\label{eqqd2}
&=\E_t\left[\mathds{1}_{\tau'<\tau=T}\right]|q-q'|+\E_t\left[\mathds{1}_{\tau=\tau'=T}\right]|q-q'|\le 2|q-q'|.
\end{align}
Using uniformly boundedness of $\E_t\left[S_{T}^{t,\x}\right]$ with respect to $t$, obtained by $\eqref{eqqGBMbound}$, we get that there exists $C>0$ independent of $t$ and of control $c$ such that 
\begin{align}
\nonumber
&\E_t\left[\left|S_{\tau}^{t,\x}Q_{\tau}^{t,\x}-S_{\tau'}^{t,\x'}Q_{\tau'}^{t,\x'}\right|\right]\le\E_t\bigg[\mathds{1}_{\{\tau'=\tau=T\}}\bigg(S_{T}^{t,\x}\underbrace{\bigg|Q_{T}^{t,\x}-Q_{T}^{t,\x'}\bigg|}_{=|q-q'|}+\underbrace{Q_{T}^{t,\x'}}_{\le q'}\bigg|S_{T}^{t,\x}-S_{T}^{t,\x'}\bigg| \bigg)\bigg]\\
\nonumber
&\qquad\qquad +\E_t\bigg[\mathds{1}_{\{\tau'<T, \tau<T\}}\bigg|S_{\tau}^{t,\x}\underbrace{Q_{\tau}^{t,\x}}_{=0}-S_{\tau'}^{t,\x'}\underbrace{Q_{\tau'}^{t,\x'}}_{=0}\bigg|\bigg]+\E_t\bigg[\mathds{1}_{\{\tau'<T=\tau\}}\bigg|S_{T}^{t,\x}\underbrace{Q_{T}^{t,\x}}_{\le q-q'}-S_{\tau'}^{t,\x'}\underbrace{Q_{\tau'}^{t,\x'}}_{=0}\bigg|\bigg]\\
\label{eqqd5}
&\qquad\le 2|q-q'|\E_t\left[S_{T}^{t,\x}\right] + \E_t\left[\left|S_{T}^{t,\x}-S_{T}^{t,\x'}\right|\right]q'\le C|q-q'|ae^{|\epsilon|}+ \E_t\left[\left|S_{T}^{t,\x}-S_{T}^{t,\x'}\right|\right]q'.
\end{align}
Using $\eqref{boundSexp}$ and merging $\eqref{eqqd2}$ and $\eqref{eqqd5}$ we get that there exists $C>0$ independent of $t$ and of control $c$ such that 
\begin{align}
\nonumber
\mathbb{E}_t&\left[\left|Q_\tau^{t,\x}\left(S_\tau^{t,\x}-\chi Q_\tau^{t,\x} \right)-Q_{\tau'}^{t,\x'}\left(S_{\tau'}^{t,\x'}-\chi Q_{\tau'}^{t,\x'} \right)\right|\right]\\
\nonumber
&\qquad\le\chi\mathbb{E}_t\left[\left|Q_\tau^{t,\x}- Q_{\tau'}^{t',\x} \right|\cdot \left|Q_\tau^{t,\x}+ Q_{\tau'}^{t',\x} \right|\right]+\E_t\left[\left|S_{\tau}^{t,\x}Q_{\tau}^{t,\x}-S_{\tau'}^{t,\x'}Q_{\tau'}^{t,\x'}\right|\right]\\
\label{eqqd11}
&\qquad\le 2\chi q|q-q'|+C|q-q'|ae^{|\epsilon|}+Cq'e^{|\epsilon'|}\left(|a-a'|+a\left|e^{\epsilon-\epsilon'}-1\right|\right).
\end{align}
Similarly to $\eqref{eqqd5}$, we get that there exists $C>0$ independent of $t$ and of control $c$ such that 
\begin{align}
\nonumber
\E_t&\left[\left|\int_t^\tau S_{r}^{t,\x}Q_{r}^{t,\x}\; dr-\int_t^{\tau'} S_{r}^{t,\x'}Q_{r}^{t,\x'}\; dr\right|\right]\\
\nonumber
&\qquad\le \E_t\left[\int_t^{\tau'} \left|S_{r}^{t,\x}Q_{r}^{t,\x}-S_{r}^{t,\x'}Q_{r}^{t,\x'}\right|\; dr\right]+\E_t\left[\mathds{1}_{\{\tau'=T\}}\cdot 0\right]+\E_t\bigg[\mathds{1}_{\{\tau'<T\}}\int_{\tau'}^\tau S_{r}^{t,\x}\underbrace{Q_{r}^{t,\x}}_{\le q-q'}\; dr\bigg]\\
\label{eqqd6}
&\qquad\le C\left(\E_t\left[\sup_{r\in[t,T]}\left|S_{r}^{t,\x}-S_{r}^{t,\x'}\right|\right]q'+|q-q'|ae^{|\epsilon|}\right)
\end{align}
and that
\begin{align}
\label{eqqd9}
\E_t&\left[\left|\int_t^\tau A_{r}^{t,\x}Q_{r}^{t,\x}\; dr-\int_t^{\tau'} A_{r}^{t,\x'}Q_{r}^{t,\x'}\; dr\right|\right]\le C\left(\E_t\left[\sup_{r\in[t,T]}\left|A_{r}^{t,\x}-A_{r}^{t,\x'}\right|\right]q'+|q-q'|a\right).
\end{align}
Using boundedness of $Q_s$ and $\eqref{boundQind}$, we get
\begin{align}
\nonumber
\E_t&\left[\left|\int_t^\tau \left(Q_{r}^{t,\x}\right)^2\; dr-\int_t^{\tau'} \left(Q_{r}^{t,\x'}\right)^2\; dr\right|\right]\le 2q\E_t\bigg[\bigg|\int_t^{\tau'} \underbrace{\left(Q_{r}^{t,\x}-Q_{r}^{t,\x'}\right)}_{=q-q'}\; dr\bigg|\bigg]+\E_t\left[\int_{\tau'}^\tau \left(Q_{r}^{t,\x}\right)^2\; dr\right]\\
\label{eqqd3}
&\qquad\le 2qT|q-q'|+\E_t\left[\mathds{1}_{\{\tau'=T\}}\cdot 0\right]+T\E_t\left[\mathds{1}_{\{\tau'<T\}}|q-q'|^2\right]\le 2T\left(q|q-q'|+|q-q'|^2\right).
\end{align}
Merging $\eqref{eqqd6}$, $\eqref{eqqd9}$ and $\eqref{eqqd3}$ and applying $\eqref{boundSexp}$ and $\eqref{boundonlipsA}$, we conclude that there exists $C>0$ independent of $t$ and of control $c$ such that 
\begin{align}
\label{eqqd13}
\mathbb{E}_t&\left[\left| \int_t^\tau \mathcal{G}_2(\X_r^{t,\x}) \;dr-\int_t^{\tau'} \mathcal{G}_2(\X_r^{t,\x'}) \;dr \right|\right]\\
\nonumber
&\qquad\le C\left( q|q-q'|+|q-q'|^2+q'\left(e^{|\epsilon'|}+1\right)\left(|a-a'|+a\left|e^{\epsilon-\epsilon'}-1\right|\right) +|q-q'|a\left(e^{|\epsilon|}+1\right)\right).
\end{align}
Finally,
\begin{align*}
\E_t\left[\left|M_{\tau}^{t,\x}-M_{\tau'}^{t,\x'}\right|\right]&\le \E_t\left[\int_t^{\tau'} \left|c_r\left(S_{r}^{t,\x}-S_{r}^{t,\x'}\right)\right|\; dr\right]+\E_t\left[\int_{\tau'}^\tau \left|c_r\left(S_{r}^{t,\x}-\eta c_r\right)\right|\; dr\right]\\
&\le\left(\E_t\left[\left(\int_t^{\tau'} c_r\; dr\right)^2\right]\right)^{\frac12}\left(\E_t\left[\sup_{r\in[t,T]}\left|S_{r}^{t,\x}-S_{r}^{t,\x'}\right|^2\right]\right)^{\frac12}\\
&\qquad+\underbrace{\left(\E_t\left[\left(\int_{\tau'}^{\tau} c_r\; dr\right)^2\right]\right)^{\frac12}}_{\le |q-q'|}\left(\E_t\left[\sup_{r\in[t,T]}\left|S_{r}^{t,\x}\right|^2\right]\right)^{\frac12}+\eta \E_t\left[\int_{\tau'}^{\tau} c_r^2\; dr\right].
\end{align*}
By using Holder's inequality with parameters $\frac {\gamma+1}\gamma$ and $1+\gamma$, we get
\begin{align*}
\E_t\left[\int_{\tau'}^{\tau} c_r^2\; dr\right]&\le \left(\E_t\left[\int_{\tau'}^\tau \left(c_r^{\frac{\gamma}{\gamma+1}}\right)^{\frac {\gamma+1}\gamma} dr\right]\right)^{\frac{\gamma}{\gamma+1}}\left(\E_t\left[\int_{\tau'}^\tau \left(c_r^{\frac{\gamma+2}{\gamma+1}}\right)^{1+\gamma}\; dr\right]\right)^{\frac{1}{1+\gamma}}\\
&\le \left(\underbrace{\E_t\left[\int_{\tau'}^\tau c_r \;dr\right]}_{\le |q-q'|}\right)^{\frac{\gamma}{\gamma+1}}\left(\underbrace{\E_t\left[\int_t^T c_r^{\gamma+2}\; dr\right]}_{\le N(1+a)^{2+\gamma}\left(1+e^{N(2+\gamma)\epsilon}\right)}\right)^{\frac{1}{1+\gamma}}.
\end{align*}
Hence, using $L^{2+\gamma}$ boundedness of process $(c_r)_{r\in [t,T]}$, for any $(c_r)_{r\in [t,T]}\in\tilde{\mathcal{A}}_{\gamma}(t,\x)\cap \tilde{\mathcal{A}}_{\gamma}(t,\x'),$
\begin{align}
\label{eqqd7}
\E_t\left[\left|M_{\tau}^{t,\x}-M_{\tau'}^{t,\x'}\right|\right]&\le q'\left(\E_t\left[\sup_{r\in[t,T]}\left|S_{r}^{t,\x}-S_{r}^{t,\x'}\right|^2\right]\right)^{\frac12}+|q-q'|\left(\E_t\left[\sup_{r\in[t,T]}\left(S_{r}^{t,\x}\right)^2\right]\right)^{\frac12}\\
\nonumber
&\qquad+\eta N^{\frac{1}{1+\gamma}}|q-q'|^{\frac{\gamma}{\gamma+1}} (1+a)^{\frac{2+\gamma}{1+\gamma}}\left(1+e^{N(2+\gamma)\epsilon}\right)^{\frac{1}{1+\gamma}}.
\end{align}
All previous inequality can also be obtained when $q\le q'$. By merging inequalities $\eqref{eqqd11}$, $\eqref{eqqd13}$ and $\eqref{eqqd7}$ into $\eqref{eqqd1}$ and using arbitrariness of control $c$ and $\eqref{boundSexp}$, we have proved $\eqref{boundloc}$.
\end{proof}

Continuity of value function $v$ is proved using Lemma \ref{lemmalocbound}. Let $(t',\x')\in[0,T]\times \mathcal{O}$ be fixed. We assume w.l.o.g. that $t\le t'$. We observe that
\begin{equation}
\label{eqqlimb}
|v(t,\x)-v(t',\x')|\le |v(t,\x)-v(t,\x')|+|v(t,\x')-v(t',\x')|.
\end{equation}
However, $|v(t,\x)-v(t,\x')|\to 0$ uniformly on $t$ for $\x\to\x'$ as stated in Lemma \ref{lemmalocbound}. If we apply Dynamic Programming Principle \cite[Remark 3.3.3]{Pham}, we get that for any $\delta>0$ there exists $c\in\tilde{\mathcal{A}}_{\gamma}(t,\x')$ such that
\begin{equation*}
|v(t,\x')-v(t',\x')|-\delta\le\E_t\left[ \left|\int_t^{t'}\mathcal{G}_2(\X_r^{t,\x'})\;dr\right| +\left|v(t',\X_{t'}^{t,\x'})-v(t',\x')\right|\right].
\end{equation*}
Using boundedness of $Q_r$ and $\eqref{eqqGBMbound}$, it is easy to show that there exists $C$ such that
$$\E_t\left[ \left|\int_t^{t'}\mathcal{G}_2(\X_r^{t,\x'})\;dr\right| \right]\le C|t-t'|^{\frac 12}.$$
By using Lemma \ref{lemmalocbound}, boundedness of $Q_r$, $L^p$-integrability of $A_r$ and $e^{\varepsilon_r}$ for any $p\ge 1$ and Holder inequality, we get that there exists $C>0$ independent of $t$ and $\delta$ such that
\begin{align}
\label{eqq49}
|v(t,\x')-v(t',\x')|\le \delta+C\Bigg(|t'-t|^{\frac12}&+\left(\E_t\left[|A_{t'}^{t,\x'}-a'|^{2}\right]\right)^{\frac 12}\\
\nonumber
&+\left(\E_t\left[\left|e^{\varepsilon_{t'}^{t,\x'}-\epsilon'}-1\right|^{2}\right]\right)^{\frac 12}+\left(\E_t\left[|Q_{t'}^{t,\x'}-q'|^{2\frac{\gamma}{\gamma+1}}\right]\right)^{\frac 12}\Bigg).
\end{align}
We observe that using Holder inequality with coefficients $2+\gamma$ and $\frac{2+\gamma}{1+\gamma}$
\begin{align}
\nonumber
\E_t\left[|Q_{t'}^{t,\x'}-q'|^{\frac{2\gamma}{\gamma+1}}\right]&\le\E_t\left[\left(\int_t^{t'}c_r\; dr\right)^{\frac{2\gamma}{\gamma+1}}\right]\le |t-t'|^{\frac{2\gamma}{2+\gamma}}\E_t\left[\left(\int_0^T c_r^{2+\gamma}\; dr\right)^{\frac {2\gamma}{(2+\gamma)(1+\gamma)}}\right]\\
\label{eqq50}
&\le |t-t'|^{\frac{2\gamma}{2+\gamma}}\left(\E_t\left[\int_0^T c_r^{2+\gamma}\; dr\right]\right)^{\frac {2\gamma}{(2+\gamma)(1+\gamma)}}.
\end{align}
Here in the last inequality we used Jensen's inequality for $\frac {2\gamma}{(2+\gamma)(1+\gamma)}\le 1$. Reminding that $c\in\tilde{\mathcal{A}}_{\gamma}(t,\x')$ and applying $\eqref{eqq3.5gen}$, $\eqref{eqqboundexp2}$, $\eqref{eqqGBMbound}$ and $\eqref{eqq50}$ to $\eqref{eqq49}$, we get that, uniformly on $c\in\tilde{\mathcal{A}}_{\gamma}(t,\x')$
\begin{align*}
\lim_{t\to (t')^-}|v(t,\x')-v(t',\x')|\le \delta.
\end{align*}
From arbitrariness of $\delta$ we conclude that previous limit converges to $0$. Continuity of $v$ follows from $\eqref{eqqlimb}$, by sending $(t,\x)\to(t',\x')$. The same results can be obtained when $t\ge t'$.

\subsection{Proof of Proposition \ref{propassu}}
Define for any $(t,\epsilon)\in[0,T]\times \R$, $g^*(t,\epsilon):= e^{\epsilon}\left( 1-g_2(t,\epsilon )\right)$. Condition $\eqref{condong2}$ is equivalent to proving that $g^*(t,\epsilon)\ge 0$ for any $(t,\epsilon)\in[0,T]\times \R$. A simple calculus on second PDE in $\eqref{HJBsplitted}$ shows that function $g^*$ satisfies the following PDE:
\begin{equation*}
\frac{\partial g^*}{\partial t}+\frac{\sigma_2^2}2 \frac{\partial^2 g^*}{\partial \epsilon^2}+\left(\rho\sigma_1\sigma_2-k\epsilon\right) \frac{\partial g^*}{\partial \epsilon}+\left(\frac{g_3(t)}{\eta} +\mu_1\right)g^*+e^{\epsilon}\left( k\epsilon-\frac{\sigma_2^2}2-\rho\sigma_1\sigma_2-\mu_1+\phi_2 \right)+\phi_3=0
\end{equation*}
on $[0,T)\times \R$, with terminal condition $g^*(T,\epsilon)=e^{ \epsilon}(1-g_2(T,\epsilon))=0$.

We check that conditions of Feynman-Kac Theorem are fulfilled for function $g^*$. As we have proved in $\eqref{g3}$, $g_3$ is twice differentiable and bounded from above and function $e^{\epsilon}\left( k\epsilon-\frac{\sigma_2^2}2-\rho\sigma_1\sigma_2-\mu_1+\phi_2 \right)$ is linearly exponential on variable $\epsilon$. Hence, we get the following Feynmann-Kac representation for $g^*$:
\begin{align}
\label{eqq22}
g^*(t,\epsilon)&=\mathbb{E}_t\left[ \int_t^T \exp\left( \int_t^r \left(\frac{g_3(s)}{\eta}+\mu_1\right)\; ds \right)\left(e^{\tilde{\varepsilon}_r}\left( k\tilde{\varepsilon}_r-\frac{ \sigma_2^2}2-\rho\sigma_1\sigma_2-\mu_1 +\phi_2\right)+\phi_3 \right)\; dr\right],
\end{align}
where $\tilde{\varepsilon}_r$ is the solution to the following SDE:
\begin{equation*}
d\tilde{\varepsilon}_r=(\rho\sigma_1\sigma_2-k\tilde{\varepsilon}_r)dr+\sigma_2dW_r, \quad\tilde{\varepsilon}_t=\epsilon.
\end{equation*}
$(\tilde{\varepsilon}_r)_{0\le r\le T}$ is an OU process and for any fixed $r\in[0,T]$, $\tilde{\varepsilon}_r$ is a normal distributed random variable with first two moments equal to
\begin{align*}
\mathbb{E}_t\left[\tilde{\varepsilon}_r\right]&=\epsilon e^{-k(r-t)}+\frac{\rho}{k}\sigma_1\sigma_2(1-e^{-k(r-t)})=\epsilon\bar{\mu}(r-t)+\frac{\rho}{k}\sigma_1\sigma_2\left(1-\bar{\mu}(r-t)\right),\\
\Var_t\left( \tilde{\varepsilon}_r \right) &=\frac{\sigma_2^2}{2k}\left(1-e^{-2k(r-t)}\right)=\bar{\sigma}(r-t)^2.
\end{align*}
Calculus on normally and log-normally distributed random variables gives
\begin{align}
\label{eqq27}
\mathbb{E}_t&\Bigg[e^{\tilde{\varepsilon}_r}\left( k\tilde{\varepsilon}_r-\frac{ \sigma_2^2}2-\rho\sigma_1\sigma_2-\mu_1+\phi_2 \right) \Bigg] \\
\nonumber
\qquad&= \left(k \bar{\mu}(r-t) \epsilon+k \bar{\sigma}(r-t)^2-\frac{\sigma_2^2}{2}-\rho\sigma_1\sigma_2\bar{\mu}(r-t)-\mu_1+\phi_2\right)e^{\bar{\mu}(r-t) \epsilon+\frac{ \bar{\sigma}(r-t)^2}2+\frac{\rho}{k}\sigma_1\sigma_2(1-\bar{\mu}(r-t))}.
\end{align}
By applying $\eqref{eqq27}$ to $\eqref{eqq22}$ we get result in $\eqref{eqq25}$.

We now prove that integral in $\eqref{eqq25}$ is nonnegative. From definition of $\bar{\mu}(r)$ and $\bar{\sigma}(r)^2$ in $\eqref{eqqmu}$ we have
\begin{equation*}
\frac{\partial \bar{\mu}(r)}{\partial r}=-k\bar{\mu}(r), \quad \frac{\partial\bar{\sigma}(r)^2}{\partial r}=-2k\bar{\sigma}(r)^2+\sigma_2^2.
\end{equation*}
$g^*$ can be written as
\begin{equation}
\label{rewritegstar}
g^*(t, \epsilon)=\int_t^T \hat{g}(r;t) f(r,\epsilon;t) \; dr,
\end{equation}
where $f$ is defined as
\begin{align*}
f(r,\epsilon;t):&=e^{\bar{\mu}(r-t)\epsilon+\frac{\bar{\sigma}(r-t)^2}{2}+\frac{\rho}{k}\sigma_1\sigma_2(1-\bar{\mu}(r-t))}\left(k\bar{\mu}(r-t)\epsilon+k\bar{\sigma}(r-t)^2-\frac{\sigma_2^2}{2}-\rho\sigma_1\sigma_2\bar{\mu}(r-t)-\mu_1+\phi_2\right)+\phi_3.
\end{align*}
To prove that $g^*(t, \epsilon)\ge 0$ for any $(t,\epsilon)\in[0,T]\times \R$, we show that, under condition $\eqref{eqq51}$, $f$ is nonnegative. Observing that
\begin{align*}
\partial_\epsilon f(r,\epsilon;t) &= e^{\bar{\mu}(r-t) \epsilon+\frac{\bar{\sigma}(r-t)^2}{2}+\frac{\rho}{k}\sigma_1\sigma_2(1-\bar{\mu}(r-t))}\bar{\mu}(r-t)\cdot\\
&\qquad\qquad\cdot\left(k\bar{\mu}(r-t) \epsilon+k\bar{\sigma}(r-t)^2-\frac{\sigma_2^2}{2}-\rho\sigma_1\sigma_2\bar{\mu}(r-t)-\mu_1+\phi_2+k \right),
\end{align*}
we get that for any $0\le t\le r\le T$, the minimum point $\epsilon^*(r;t)$ of $f(r,\epsilon;t)$ satisfies the following equation
\begin{align*}
\bar{\mu}(r-t)\epsilon^*(r;t)+\bar{\sigma}(r-t)^2-\frac{\sigma_2^2}{2k}-\frac{\rho}{k}\sigma_1\sigma_2\bar{\mu}(r-t)-\frac{\mu_1}{k}+\frac{\phi_2}{k}+1=0.
\end{align*}
By evaluating $f(r,\epsilon;t)$ in $\epsilon^*(r;t)$,
\begin{equation*}
f(r,\epsilon^*(r;t);t)= -ke^{-\frac{\bar{\sigma}(r-t)^2}{2}+\frac{ \sigma_2^2}{2k}-1+\frac{\mu_1}{k}+\frac{\rho}{k}\sigma_1\sigma_2-\frac{\phi_2}{k}}+\phi_3.
\end{equation*}
If condition $\eqref{eqq51}$ is satisfied, then $f(r,\epsilon^*(r;t);t)\ge 0$ and from $\eqref{rewritegstar}$ we conclude that $g^*$ is nonnegative. \qed

\subsection{Proof of Proposition \ref{propcstar}}
Using boundedness of $Q_r$ and $g_3$, linearity of $g_2$ in $a$ and linear exponential growth of $g_2$ with respect to $\epsilon$, we conclude there exist $C>0$ such that
\begin{align}
\label{1eqqq22}
|c^*(r,\X_r)|&\le C (1+q)(1+A_{r})(1+e^{C\varepsilon_r}).
\end{align} 
Applying Holder inequality, we get there exists $C_1>0$ such that
\begin{align*}
&\E_t\left[\int_0^T |c^*(r,\X_r)|^{2+\gamma} dr\right]\le C_1\left(1+q\right)^{2+\gamma}\left(1+\E_t\left[\sup_{r\in[0,T]}A_{r}^{2(2+\gamma)}\right]\right)^{\frac 1{2}}\left(1+\E_t\left[\sup_{r\in[0,T]}e^{2C(2+\gamma) \varepsilon_{r}}\right]\right)^{\frac 1{2}}.
\end{align*}
Using  $\eqref{eqqGBMbound}$, we conclude that there exists $C_2>0$, independent of $t$ and $\x$, such that
\begin{equation*}
\left(\E_t\left[\int_t^T |c^*(r,\X_r)|^{2+\gamma} dr\right]\right)^{\frac 1{2+\gamma}}\le C_2 \left( 1+a\right)\left( 1+e^{C_2|\epsilon|}\right) , 
\end{equation*}
which implies $(c^*(r,\X_r))_{r\in [t,T]}\in \tilde{\mathcal{A}}_{\gamma,C_2}(t,\x)$. \qed


\begin{thebibliography}{9}

\bibitem{Almgren2012} Almgren, R. (2012). \textit{Optimal Trading with Stochastic Liquidity and Volatility}. SIAM Journal on Financial Mathematics, 3(1), 163-181.

\bibitem{Almgren1999} Almgren, R., Chriss, N. (1999). \textit{Value under liquidation.} Journal of Risk, 12(12), 61-63.

\bibitem{Almgren2001} Almgren, R., Chriss, N. (2001). \textit{Optimal execution of portfolio transactions.} Journal of Risk, 3(2), 5-39.

\bibitem{Almgren2003} Almgren, R. (2003). \textit{Optimal execution with nonlinear impact functions and trading-enhanced risk.} Applied Mathematical Finance, 10(1), 1-18.

\bibitem{Bertsimas1998} Bertsimas, D., Lo, A. W. (1998). \textit{Optimal control of execution costs.} Journal of Financial Markets, 1(1), 1-50.

\bibitem{Davis} Brockwell, P., Davis, R., Deveaux, R., Fienberg, S. (2016). \textit{Introduction to Time Series and Forecasting.} Springer.

\bibitem{Cartea} Cartea, A., Jaimungal, S., Penalva, J.  (2015). \textit{Algorithmic and High-Frequency Trading.} Cambridge University Press.

\bibitem{Huberman2005} Huberman, G., Stanzl, W. (2005). \textit{Optimal liquidity trading.} Review of Finance, 9(2), 165-200.

\bibitem{Kar} Karatzas, I., Shreve, S. E. (1991). \textit{Brownian motion and stochastic calculus.} Springer.

\bibitem{Kharroubi} Kharroubi, I., Pham, H. (2010). \textit{Optimal Portfolio Liquidation with Execution Cost and Risk.} SIAM Journal on Financial Mathematics, 1(1), 897-931.

\bibitem{Kry} Krylov, N., Balakrishnan, A. V. (1980). \textit{Controlled diffusion processes.} Springer.

\bibitem{LiZheng} Li, Y., Zheng, H. (2015). \textit{Weak Necessary and Sufficient Stochastic Maximum Principle for Markovian Regime-Switching Diffusion Models.} Journal of Applied Mathematics and Optimization, 71(1), 39-77.

\bibitem{Obizhaeva2013} Obizhaeva, A., Wang, J. (2013). \textit{Optimal trading strategy and supply/demand dynamics.} Journal of Financial Markets, 16(1), 1-32.

\bibitem{Pham} Pham, H. (2010). \textit{Continuous-time Stochastic Control and Optimization with Financial Applications.} Springer.

\bibitem{Jentzen} Weinan, E., Jiequn H., Arnulf, J. (2017). \textit{Deep Learning-Based Numerical Methods for High-Dimensional Parabolic Partial Differential Equations and Backward Stochastic Differential Equations.} Communications in Mathematics and Statistics. Springer, 5 (4), 349-380.
\end{thebibliography}
\end{document}